\documentclass{amsart}
 
\usepackage[english]{babel}
\usepackage{amssymb}
\usepackage{mathrsfs}
\usepackage{tikz-cd}
\usepackage[colorlinks]{hyperref}
\hypersetup{citecolor=teal, linkcolor=violet, urlcolor=magenta}
\usepackage{indentfirst}
\usepackage{enumerate}
\usepackage{float}
\usepackage{graphicx}
\usepackage{color}
\usepackage{mathtools}
\usepackage{caption}
\usepackage{stmaryrd}
\usepackage{multicol}
\usepackage[all]{xy}

\usepackage{nicematrix}

\usepackage{verbatim}

\usepackage[colorinlistoftodos]{todonotes}

\usepackage{xfrac}

\newtheorem{thm}{Theorem}[section]
\newtheorem{rmk}[thm]{Remark}

\newtheorem{prop}[thm]{Proposition}
\newtheorem{cor}{Corollary}[thm]
\newtheorem{lema}[thm]{Lemma}

\theoremstyle{definition}

\newtheorem{defi}[thm]{Definition}
\newtheorem{exe}[thm]{Example}
\newtheorem{que}[thm]{Question}

\usepackage{apptools}

\def\Proj{\text{Proj}}

\def\C{\mathbb{C}}

\def\P{\mathbb{P}}
\def\Q{\mathbb{Q}}

\def\N{\mathcal{N}}
\def\calP{\mathcal{P}}
\def\calQ{\mathcal{Q}}
\def\calL{\mathcal{L}}

\def\X{\mathcal{X}_{k,d,n}}
\def\H{\text{H}}

\DeclareMathOperator{\SL}{SL}
\DeclareMathOperator{\ord}{ord}

\DeclareMathOperator{\lct}{lct}
\DeclareMathOperator{\rank}{rank}

\everymath{\displaystyle}

\emergencystretch=\maxdimen
\usepackage{soul}

\title{Nets of quadric surfaces and plane cubics and their GIT stability}
\date{\today}

 \author[M. Hattori]{Masafumi Hattori \textsuperscript{1}}
\address{\textsuperscript{1} School of Mathematical Sciences, University of Nottingham, Nottingham NG7 2RD, UK.}
\email{masafumi.hattori@nottingham.ac.uk}

\author[T. S. Papazachariou]{Theodoros Stylianos Papazachariou \textsuperscript{2}}
\address{\textsuperscript{2} Yau Mathematical Sciences Center, Jingzhai, Tsinghua University, Haidian District, Beijing, China.}
\email{tpapazachariou@mail.tsinghua.edu.cn}

\author[A. Zanardini]{Aline Zanardini \textsuperscript{3}}
\address{\textsuperscript{3} Institut de Mathématiques, 
École Polytechnique Fédérale de Lausanne, Lausanne, Switzerland.}
\email{aline.zanardini@epfl.ch}

\begin{document}

\begin{abstract}
     A general net of quadric surfaces, together with a choice of a base point, defines a net of plane cubics via the Gale transformation of the remaining seven base points. To both nets, one can also naturally associate the same smooth plane quartic. In this paper, we generalize the cycle of correspondences arising from nets of quadrics that define rational elliptic threefolds and provide a complete criterion for GIT stability of the three underlying geometric objects using birational-geometric techniques. 
\end{abstract}

\maketitle

\setcounter{tocdepth}{3}
{\hypersetup{hidelinks}
\tableofcontents}

\section{Introduction}

Even though Mumford's geometric invariant theory (GIT) is one of the fundamental tools in the construction of moduli spaces in algebraic geometry, explicit criteria for GIT (semi)stability remain relatively rare. This paper presents a birational-geometric method to address three classical, apparently unrelated GIT problems explicitly. More precisely, we relate the GIT stability of nets of quadrics in $\P^3$, of nets of plane cubics, and of plane quartics. Our approach is based on a criterion that stems from analyzing singularities of log pairs in the spirit of the previous work by the first and third authors \cite{hz} and on showing that three log pairs naturally associated with the three objects in consideration are log crepant and related by a log Calabi-Yau fibration (cf.~\cite{ambro2004shokurov}). Our main result shows, through a birational-geometric framework, that there is an equivalence, among GIT stability of nets of quadrics in $\P^3$ with a discriminant quartic $\Delta(\N)$ with only ADE singularities, of nets of plane cubics with seven base points, and of the plane quartics $\Delta(\N)$ (Theorem \ref{thm:main}). This can be summarized as follows. 

\begin{thm}\label{thm:main}
    Let $\N$ be a net of quadrics in $\P^3$ with a discriminant quartic $\Delta(\N)$ that is reduced and has only ADE singularities. Then, any choice of a base point $p \in \P^3$ of $\N$ defines a net of plane cubics $G(\N,p)$ with seven base points.
    
    Moreover, the following statements about the GIT (semi)stability of $\N, G(\N,p)$ and $\Delta(\N)$ are equivalent.
    \begin{enumerate}[(i)]
        \item $\N$ is GIT stable (resp. semistable) with respect to the natural action of $\SL(4)$.
        \item $G(\N,p)$ is GIT stable (resp. semistable) with respect to the natural action of $\SL(3)$.
        \item $\Delta(\N)$ is GIT stable (resp. semistable) with respect to the natural action of $\SL(3)$.
    \end{enumerate}
\end{thm}

 A more precise description of the three GIT problems is provided in Section \ref{3git}, and a slightly stronger statement that depends on introducing further notations will be the content of Theorem \ref{thm--equiv.good.nets}.
 Note that the complete description of the GIT stability of the plane quartics $\Delta(\N)$ is well known and can be found, for example, in  \cite{mumford,ascher2024wall}.
 
 We also establish the following useful criterion, which holds without assuming that $\Delta(\N)$ has only ADE singularities.

\begin{cor}[= Corollary \ref{cor--final}]\label{cor:main}
Let $\mathcal{N}$ be a net of quadrics in $\mathbb{P}^3$.
Then, $\mathcal{N}$ is GIT stable (resp.~semistable) if and only if $\Delta(\mathcal{N})\neq \P^2$ and $\Delta(\mathcal{N})$ is also GIT stable (resp.~semistable). 
    \end{cor}

The correspondence $\N\mapsto G(\N,p)$, which is defined using the projection of lines passing through $p$, is a generalization of a manifestation of Gale duality, a classical association between collections of points in projective spaces, and is detailed in Section \ref{sec:construction_extended} and Proposition \ref{qtoc}. To our knowledge, such correspondence was known only for nets with sufficiently general base points (see Section \ref{sec:cycle}), and we extend it to a significantly larger class of degenerate configurations. 

Let us now present a more geometric characterization of the nets to which Theorem \ref{thm:main} applies. Such nets were first considered in \cite{extremal} and are called \textbf{good} in this paper (see Definitions \ref{def: regular} and \ref{def:good}).

\begin{prop}\label{prop:main}
Let $\N$ be a net of quadrics in $\P^3$. Then the following are equivalent:
\begin{enumerate}
\item $\N$ is {\bf good}, that is, $\N$ has exactly eight base points (possibly infinitely near), and there exist two distinct members $Q_1$ and $Q_2$ of $\mathcal{N}$ such that $Q_1\cap Q_2$ is smooth at any $\P^3$-base point of $\mathcal{N}$;
    \item $\N$ has exactly eight base points (possibly infinitely near) and for any base point $p\in\P^3$ of $\N$, at most one quadric of $\N$ is singular at $p$;
    \item $\N$ has eight base points $\{p_1,\ldots,p_8\}$ (possibly infinitely near) such that $\N$ defines an elliptic fibration $\mathrm{Bl}_{p_1,\ldots,p_8}(\P^3)\to\P^2$;
    \item $\Delta(\N)$ is a reduced plane curve with only ADE singularities.
\end{enumerate}
\end{prop}

The first two conditions were observed to be equivalent and imply the third condition in \cite[Assumption 1]{extremal}. 
We show that the converse also holds in Proposition \ref{prop--ellip-fibration--criterion}.
Furthermore, we prove the equivalence of the fourth condition with the three others in Propositions \ref{prop:good--ADE--corresp} and \ref{extended}.

\vspace{2mm}
\noindent{\bf Main idea of the proof of Theorem \ref{thm:main}.} 
To prove Theorem \ref{thm:main}, we apply the ideas first developed in \cite{hz}.
In \cite{hz}, the authors showed that the GIT stability of a $k$-dimensional linear system $\calL$ of degree $d$ in $\P^n$ can be detected by analyzing the singularities of the log pairs $(\P^n,\tfrac{n+1}{d(k+1)}(H_1+\ldots+H_{k+1}))$ with respect to toric valuations (Definitions \ref{defi--log--pair} and \ref{DefToricDiv}), where $H_1,\ldots,H_{k+1}$ are hypersurfaces that generate $\calL$ (cf.~Theorem \ref{thm:stab_linear_systems} and Lemma \ref{hypersurftoric}).
We therefore reduce the analysis of the GIT (semi)stability of each of the three types of objects (nets of quadric surfaces, nets of plane cubics, and plane quartics) that we consider here to an analysis of the singularities of three kinds of log pairs with respect to toric valuations. 

\vspace{2mm}
\noindent{\bf The $G(\N,p)$--$\Delta(\N)$ correspondence.} For example, if we consider the net of cubics $G(\N,p)$, we can then define a weak del Pezzo surface $Y$ as the blowing up of $\P^2$ along the base locus of $G(\N,p)$. 
Then we can show that $Y$ is a double cover of $|\N|(\cong\P^2)$ branched along $\Delta(\N)$ (Proposition \ref{extended}).
In this setup, GIT stability of $G(\N,p)$ is detected by the singularities of the pair $(Y, \tfrac{1}{3}(\tilde{C}_1+\tilde{C}_2+\tilde{C}_3))$ with respect to toric valuations, where $\tilde{C}_1$, $\tilde{C}_2$ and $\tilde{C}_3$ are the curves corresponding to a choice of generators $C_1$, $C_2$ and $C_3$ of $G(\N,p)$ such that $(\P^2,\tfrac{1}{3}(C_1+C_2+C_3))$ and $(Y, \tfrac{1}{3}(\tilde{C}_1+\tilde{C}_2+\tilde{C}_3))$ are log crepant.
Letting $H_i$ be the lines in $|\N|$ corresponding to the cubics $C_i$, we see that  $(Y, \tfrac{1}{3}(\tilde{C}_1+\tilde{C}_2+\tilde{C}_3))$ and $(|\N|,\tfrac{1}{3}(H_1+H_2+H_3)+\tfrac{1}{2}\Delta(\N))$ are log crepant in the sense of Definition \ref{def:logcrep} (Proposition \ref{prop--log--crepant}) and therefore we can compare their GIT stability.

Regarding the outlined strategy, we also prove the following auxiliary results, which enable our comparison strategy. 

\begin{thm}[=Theorem \ref{thm--cubics}]\label{thm:main2}
    Fix an integer $n\geq 2$ and let $C_1,\ldots,C_n$ be pairwise distinct cubic curves in $\mathbb{P}^2$. Then the log pair $\left(\mathbb{P}^2,\tfrac{1}{n}(C_1+\ldots+C_n)\right)$ is lc if and only if the degree $3n$ curve $C_1+\ldots+C_n$ is GIT semistable with respect to the natural action of $
    \SL(3)$.
Furthermore, $C_1+\ldots+C_n$ is GIT stable if and only if the log pair $\left(\mathbb{P}^2,\tfrac{1}{n}(C_1+\ldots+C_n)\right)$ is klt or lc with only one lc center that is a smooth conic on $\P^2$.
\end{thm}

\begin{cor}[= Corollary \ref{cor--cubic--linear--sys}]\label{cor:main2}
    Let $\calL$ be a $k$-dimensional linear system of plane cubic curves for any $k\ge1$.
    Then, $\calL$ is GIT (semi)stable with respect to the natural action of $
    \SL(3)$ if and only if for any choice of generators $C_1,\ldots, C_{k+1}\in\mathcal{L}$ the log pair $\left(\mathbb{P}^2, \tfrac{1}{k+1}(C_1+\ldots+C_{k+1})\right)$ is klt (resp.~lc).
\end{cor}

We note that the above corollary is in fact a combination of Theorem \ref{thm:main2} and \cite[Theorem 1.1]{hz} (see also Theorem \ref{thm:stab_linear_systems}). When $k=1$, it provides an alternative description of Miranda's result \cite{stabMiranda} on GIT stability of pencils of cubic curves. When $k=2$, we see that the GIT stability of $G(\N,p)$ coincides with the log canonicity of the log pair $(\P^2,\tfrac{1}{3}(C_1+C_2+C_3))$ for any choice of generators $C_1,C_2$ and $C_3$ for the net.

\vspace{2mm}
\noindent{\bf The $G(\N,p)$--$\N$ correspondence.} To compare the GIT stability of $\N$ and $G(\N,p)$, we further use a result on log Calabi-Yau fibrations proved in \cite{ambro2004shokurov} (Proposition \ref{prop:cubics-to-quadrics:git}).
The key observation is that if we choose divisors $Q'_1$, $Q'_2$ and $Q'_3$ in  $\mathrm{Bl}_p(\P^3)$ that correspond to a choice of generators $Q_1$, $Q_2$ and $Q_3$ for $\N$, where $(\mathrm{Bl}_p(\P^3),\tfrac{2}{3}(Q'_1+Q_2'+Q'_3))$ and $(\P^3,\tfrac{2}{3}(Q_1+Q_2+Q_3))$ are log crepant, and we consider generators $C_{ij}$ for $G(\N,p)$ defined as the image of $Q'_{i}\cap Q'_j$, then the projection of lines $(\mathrm{Bl}_p(\P^3),\tfrac{2}{3}(Q'_1+Q_2'+Q'_3))\to \mathbb{P}^2$ is a log Calabi-Yau fibration (i.e.~a lc-trivial fibration in the sense of \cite{ambro2004shokurov}) such that the moduli $\mathbb{Q}$-divisor is zero and the discriminant $\mathbb{Q}$-divisor coincides with $\tfrac{1}{3}(C_{12}+C_{23}+C_{31})$. This means that the log canonicity of $(\mathrm{Bl}_p(\P^3),\tfrac{2}{3}(Q'_1+Q_2'+Q'_3))$ and that of $(\P^2,\tfrac{1}{3}(C_{12}+C_{23}+C_{31}))$ are the same, which helps us to compare their GIT stability. In particular, we obtain the following result deduced from Theorems \ref{thm:main} and \ref{thm--equiv.good.nets} and Corollary \ref{cor:main2}, which provides a relationship between the log canonicity and GIT stability of good nets of quadric surfaces and the associated nets of cubic curves.

\begin{thm}\label{thm:main3}
 Let $\N$ be a good net of quadrics in $\P^3$ and let $p$ be a $\P^3$-base point.
 Then the conditions of Theorem \ref{thm:main} are also equivalent to the following.
 \begin{itemize}
     \item[{\it (iv)}] Let $Q_1$ be a general element of $\mathcal{N}$. Then for any choice of two other generators $Q_2$ and $Q_3\in \mathcal{N}$, $(\mathbb{P}^3,\tfrac{2}{3}(Q_1+Q_2+Q_3))$ is klt (resp.~lc).
     \item[{\it (v)}] 
     For any choice of three generators $C_1$, $C_2$ and $C_3\in G(\mathcal{N},p)$, $(\mathbb{P}^2,\tfrac{1}{3}(C_1+C_2+C_3))$ is klt (resp.~lc).
 \end{itemize}
\end{thm}

In contrast, to obtain Corollary \ref{cor:main}, we also have to deal with nets whose discriminant is a double conic, which are not good by Proposition \ref{prop:main}, and we are not able to use the types of birational arguments described so far. Instead, we use a complete classification of GIT unstable nets of quadrics (Theorem \ref{unstable nets}) obtained by the algorithm in \cite{Pap22} and the computational code \cite{Pap_code} developed by the second-named author (see Appendix \ref{section--Theo's-calculation}).

\subsection*{Organization of the paper} We now proceed to explain how the paper is organized. In Section \ref{sec:background}, we present the necessary background concepts and establish the notations used throughout the paper. In Section \ref{sec:construction_extended}, we extend the classical threefold cycle of correspondences alluded to in the abstract and explained in Section \ref{sec:cycle}. Section \ref{sec:git_cubics} is devoted to the proof of Theorem \ref{thm:main2} and Section \ref{sec:new_results} to the proof of our main result. In Section \ref{sec:questions}, we collect some questions to be addressed in future work that naturally arise from this project. Finally, Appendix \ref{section--Theo's-calculation} describes all unstable nets of quadric surfaces in $\P^3$ geometrically and studies their corresponding discriminant quartics.

\subsection*{Acknowledgments}
The authors thank Yuchen Liu and Junyan Zhao for helpful comments regarding Remark \ref{rmk-lz}.
The authors would like to thank the Isaac Newton Institute for Mathematical Sciences, Cambridge, for support and hospitality during the programme New Equivariant Methods in Algebraic and Differential Geometry, where work on this paper was undertaken. 
This work was supported by EPSRC grant EP/R014604/1.
M.H.~was 
supported by Royal Society International Collaboration Award ICA\textbackslash1\textbackslash231019. T.P.~was supported by a Shuimu Scholarship from Tsinghua University.

\section{Background and notations}\label{sec:background}

This section presents the relevant background material needed for our exposition. We work over the field of complex numbers $\mathbb{C}$ and use the following notations throughout this paper.

\begin{enumerate}
[({N}\arabic{enumi})]
\item We call $X$ a {\it variety} if $X$ is a quasi-projective irreducible and reduced scheme over $\mathbb{C}$. Unless otherwise stated, we call a closed point $p\in X$ a point for simplicity. 
\item $\mathbb{M}_{r\times s}$ denotes the set of $r\times s$-matrices of complex numbers.
\item $\P^{r,s}$ denotes the space parameterizing constellations of infinitely near points associated to a zero-dimensional subscheme of $\P^r$ of length $s$. More precisely, up to ordering, each point $\calP\in\P^{r,s}$ is represented by a set 
\[
S_{\calP}\coloneqq\{p_1^{(1)},\ldots,p_1^{(m_1)},\ldots,p_{\ell}^{(1)},\ldots,p_{\ell}^{(m_{\ell})}\},
\]
where $m_1+\ldots+m_{\ell}=s$, each $p_j^{(1)}$ is a point in $\P^r$ and each $p_j^{(i+1)}$ is infinitely near to the previous point $p_j^{(i)}$ (of order one), that is, $p_j^{(i+1)}$ is a point on the exceptional divisor of the blowing up at $p_j^{(i)}$. For simplicity, we will also write $p_j$ instead of $p_j^{(1)}$.
\item Given $\calP\in\P^{r,s}$, we denote by $\calL_d(\calP)$ the linear system of the hypersurfaces of degree $d$ in $\P^r$ whose base locus is given by $S_{\calP}$. Moreover, we refer to the points $p_1, 
\ldots,p_{\ell}$ as {\it $\P^r$-base points} or {\it base points in $\P^r$} of $\calL_d(\calP)$. \label{n2}
\item We denote the space of all $k$-dimensional linear systems of hypersurfaces of degree $d$ in $\P^n$ by $\X$, which can be regarded as a Grassmannian variety.\label{n3}
\item Given a point $\mathcal{L}\in \mathcal{X}_{k,d,n}$ we often write $|\calL|$ to emphasize that we are considering a linear system -- a projective space parameterizing the lines through the origin of some $k+1$-dimensional linear subspace of $\H^0(\P^n,\mathcal{O}_{\P^n}(d))$.  
Then $V_{\calL}\subset \H^0(\P^n,\mathcal{O}_{\P^n}(d))$ denotes the corresponding vector space and $|\mathcal{L}|^{\vee}$ will denote the dual projective space parameterizing hyperplanes. 
Furthermore, we say that $k+1$-elements $H_1,\ldots,H_{k+1}\in |\calL|$ {\it generate} $\calL$ if choosing $f_i\in V_{\calL}$ that represents $H_i$ for each $i=1,\ldots,k+1$, we have that $V_{\calL}=\text{Span}(f_1,\ldots,f_{k+1})$. \label{n4} 
\item We say that a $k$-dimensional linear system $\calL$ is a {\it net} (resp.~{\it pencil}) if $k=2$ (resp.~$k=1$).
\end{enumerate}

\subsection{The relevant geometric invariant theory}\label{sec:git}

We start by recalling some basic definitions and results from geometric invariant theory (GIT). For more details, see \cite{mukai}, \cite{GIT}, or \cite{hoskins_GIT}. 

Let $X$ be a projective variety on which the group $ \SL(n+1)$ acts and let $\mathscr{A}$ be an ample $ \SL(n+1)$-linearized line bundle on $X$. If $\H^0(X, \mathscr{A}^{\otimes k})^{\SL(n+1)}$ denotes subspace of all $\SL(n+1)$-invariant sections of $\H^0(X, \mathscr{A}^{\otimes k})$, then the projective variety $$\Proj\bigoplus_{k\in\mathbb{Z}_{\ge0}}\H^0(X, \mathscr{A}^{\otimes k})^{\SL(n+1)}$$ is called the {\it GIT quotient of $X$ with respect to the $\SL(n+1)$-action} and is denoted by $X\sslash \SL(n+1)$ in this paper. This variety $X\sslash \SL(n+1)$ is the categorical quotient of the locus of \textit{semistable} points by $\SL(n+1)$, where a semistable point is defined as follows.

\begin{defi}\label{def:stab}
 A closed point $x\in X$ is \textit{GIT semistable with respect to $\mathscr{A}$} if there exist some $m\in\mathbb{Z}_{>0}$ and an $\SL(n+1)$-invariant section $s\in \H^0(X,\mathscr{A}^{\otimes m})$ such that $s(x)\neq 0$. In addition, if the $ \SL (n+1)$-orbit of $x$ is closed with a finite stabilizer group, then we say that the point $x$ is \textit{GIT stable with respect to $\mathscr{A}$}. We say that a GIT semistable point $x$ with closed $ \SL (n+1)$-orbit is {\it GIT polystable with respect to $\mathscr{A}$}.
\end{defi}

When the projective variety $X$ has Picard number one, the above notions of stability are independent of the choice of $\mathscr{A}$, and we say that $x$ is (semi/poly)stable (see e.g., \cite[Proposition 1.4, Corollary 1.20]{GIT}). If we denote the open locus of semistable (resp.~stable) points by $X^{\mathrm{ss}}\subset X$ (resp.~$X^{\mathrm{s}}\subset X$), then points in $X \backslash X^{\mathrm{ss}}$ (resp. $X \backslash X^{\mathrm{s}}$) will be called {\it unstable} (resp. {\it not stable}). Moreover, taking two points $x,y\in X^{\mathrm{ss}}$ and letting $O_x$ and $O_y$ be the orbits of $x$ and $y$, respectively, as in the moduli theory of vector bundles, we say that $x$ and $y$ are {\it S-equivalent} if $\overline{O_x}\cap\overline{O_y}\ne\emptyset$, where $\overline{O_x}$ and $\overline{O_y}$ are the Zariski closures. 

In general, we can detect (semi)stability of a point $x\in X$ by a numerical invariant known as the Hilbert-Mumford weight, which can be defined as follows. We say that $\lambda$ is a {\it one-parameter subgroup} of a group $G$ if $\lambda: \C^{\times}\to G$ is a homomorphism of algebraic groups.
Given any one-parameter subgroup $\lambda: \C^{\times}\to \SL(n+1)$ and a fixed (not necessarily ample) $\SL(n+1)$-linearized line bundle $\mathscr{M}$ on $X$, the corresponding {\it Hilbert-Mumford weight} of $x\in X$ with respect to $\lambda$ and $\mathscr{M}$ is the quantity
 \[
 \mu^{\mathscr{M}}(x,\lambda)\coloneqq -\textrm{the weight of the action of }\mathbb{C}^\times \textrm{ on }\mathscr{M}\otimes k(y) \textrm{ via }\lambda,
 \]
  where $y$ is the unique point in $X$ such that $y=\lim_{t\to0}\lambda(t)\cdot x$, and $k(y)$ is the corresponding residue field. The Hilbert-Mumford criterion \cite[Theorem 2.1]{GIT} states that a closed point $x\in X$ is GIT semistable (resp. stable) with respect to an ample $ \SL (n+1)$-linearized line bundle $\mathscr{A}$ if and only if  $\mu^{\mathscr{A}}(x,\lambda)\geq 0$ (resp. $\mu^{\mathscr{A}}(x,\lambda)>0$) for any non-trivial one-parameter subgroup $\lambda: \C^{\times}\to \SL(n+1)$.
  
We finally note that given $\lambda$, there exists a maximal torus $T$ of $\SL(n+1)$ such that $\lambda$ factors through $T$ and all maximal tori are conjugate to each other.
We write  $\lambda=\mathrm{Diag}(\alpha_0,\ldots,\alpha_n)$ if  $\lambda$ factors through the diagonal torus and 
\[
\lambda(t)=\begin{pmatrix}
           t^{\alpha_0} & 0 & \ldots & 0 \\ 
           0 & t^{\alpha_1} & \ldots & 0 \\ 
           \vdots & \vdots & \ddots & \vdots\\ 
           0 & 0 & \ldots & t^{\alpha_n}
        \end{pmatrix}.
\]

\subsubsection{GIT for linear systems of hypersurfaces}\label{3git}

Fix now a positive integer $n$ and let $V$ denote the $n+1$-dimensional vector space $\H^0(\mathbb{P}^n,\mathcal{O}_{\mathbb{P}^n}(1))$. Then for each integer $d\geq 1$, the projective space $\P(S^{d}V^{\vee})$ parametrizes hypersurfaces of degree $d$ on $\P(V)=\P^n$. In particular, the space $\X$ of $k$-dimensional linear systems of hypersurfaces of degree $d$ in $\P^n$ can be embedded in the projective space $\mathbb{P}^{N}=\mathbb{P}(\Lambda^{k+1} S^{d}V^{\vee})$ via Pl\"{u}cker coordinates. The natural action of the group $G=\SL(V)$ on $V\simeq \C^{n+1}$ induces an action on this large projective space, hence on the invariant subvariety $X=\X$. 
We set $\mathscr{A}:=\mathcal{O}_{\P^N}(1)|_{\X}$ and let $\mu(x,\lambda)$ denote $\mu^{\mathscr{A}}(x,\lambda)$ for any $x\in \X$ and one-parameter subgroup $\lambda$ for simplicity.
The associated GIT problem was addressed in  \cite{hz} and \cite{Pap22, compint}.

We recall the subsequent general criterion.

\begin{thm}[{\cite[Theorem 1.1]{hz}}]\label{thm:stab_linear_systems}
    A linear system $\calL \in \X$ is GIT stable (resp. semistable) if and only if for any choice of generators $H_1,\ldots, H_{k+1}\in \calL$ the degree $d(k+1)$ hypersurface $H_1+\ldots+H_{k+1}$, viewed as an element of $\mathbb{P}(S^{d(k+1)}V^{\vee})$, is GIT stable (resp. semistable).
\end{thm}

Concerning GIT polystability, we can further show the following, which is of independent interest and complements the above criterion.

\begin{thm}
Let $\calL\in \X$ and denote by $G_{\calL}$ its stabilizer subgroup under the action of $\SL(n+1)$.  Choose $T$ a maximal subtorus of $G_{\calL}$ and $H_{1}$ a $T$-invariant hypersurface in $\calL$. Suppose further that $\calL$ is GIT semistable. Then $\calL$ is GIT polystable if and only if the hypersurface $H_{1}+H_{2}+\ldots+H_{k+1}$ is GIT polystable for any choice of $T$-invariant elements $H_{2},\ldots,H_{k+1}$ in $\calL$ such that $H_1,H_2,\ldots,H_{k+1}$ generate $\calL$.
\end{thm}

\begin{proof}
Let $Z$ denote the centralizer of $T$ in $\SL(n+1)$. 
Assume that $\calL$ is GIT semistable but not polystable under the condition that $H_{1}+H_{2}+\ldots+H_{k+1}$ is GIT polystable for any $T$-invariant elements $H_{2},\ldots,H_{k+1}$ in $\calL$ such that $H_{1},H_{2},\ldots,H_{k+1}$ generate $\calL$. Then the orbit $\SL(n+1)\cdot \calL$ is not closed in $\X$.
Given a choice of $H\coloneqq H_{1}+H_{2}+\ldots+H_{k+1}$, as above, let $G_H$ be the corresponding stabilizer and note that $G_H$ also contains $T$ as a maximal subtorus. Indeed, for any one-parameter subgroup $\lambda$ that fixes $H$, we see that $\lambda$ also fixes each $H_{j}$ and hence $\mathcal{L}$.
Now, since $\calL$ is not GIT polystable, by \cite[Corollary 4.5]{kempf1978}, there exists a one parameter subgroup $\lambda$ of $Z$ such that $\calL^{\infty}:=\lim_{t\to0}\lambda(t)\cdot \calL$ is GIT semistable. Hence, $\mu(\calL,\lambda)=0$, but $\lambda$ is not contained in $G_{\calL}$.  By the proof of \cite[Lemma 3.2]{hz}, we can replace $H_{2},\ldots,H_{k+1}$ and may assume that $\mu(H,\lambda)=0$. Thus, setting $H^{\infty}_{j}:=\lim_{t\to0}\lambda(t)\cdot H_{j}$, we have that the $H^{\infty}_{i}$ generate $\calL^{\infty}$.  Moreover, $H^{\infty}\coloneqq H_{1}^\infty+\ldots+H_{k+1}^\infty=g\cdot H$ for some $g\in \SL(n+1)$ since $H$ is GIT polystable.
This means that $gG_Hg^{-1}$ is the stabilizer of $H^\infty$. Then the torus generated by $T$ and $\lambda$ fixes $H^{\infty}$. This contradicts the fact that $T\subset G_H$ is a maximal torus. 

Next, we deal with the converse statement. 
Assume that $\calL$ is GIT polystable and there exists a choice of $T$-invariant hypersurfaces $H_{2},\ldots,H_{k+1}$ in $\calL$ such that $H_{1},H_{2},\ldots,H_{k+1}$ generate $\calL$, and $H$ defined as above is GIT semistable but not polystable. Then, $T$ is a maximal subtorus of $G_H$ again, and by \cite[Corollary 4.5]{kempf1978}, we can take a one-parameter subgroup $\lambda$ of $Z$ that is not contained in $G_H$ with $\mu(H,\lambda)=0$ as in the first paragraph.
Since $\mu(H,\lambda)\ge \mu(\calL,\lambda)$ (cf.~\cite[Lemma 3.1]{hz}) and $\calL$ is GIT semistable, we obtain that $\mu(\calL,\lambda)=0$. Now, let $G_{\infty}$ be the stabilizer group of  $\calL^\infty:=\lim_{t\to 0}\lambda(t)\cdot \calL$. Then, since $\calL$ is GIT polystable, $G_{\calL}\cong G_\infty$. However, we see that each $H_j^\infty:=\lim_{t\to 0}\lambda(t)\cdot H_{j}$ is fixed by both $T$ and $\lambda$. Thus, $\calL^\infty$ is fixed by $T$ and $\lambda$. Then the dimension of a maximal torus of $G_{\infty}$ is strictly larger than that of $T$. This is a contradiction.
\end{proof}

In Section \ref{sec:new_results}, we investigate some concrete links among three specific instances of the described GIT problem for linear systems of hypersurfaces. More precisely, we will relate the (semi)stability of certain:

\begin{enumerate}[(A)]
    \item nets of quadric surfaces in $\P^3$, which corresponds to $(k,d,n)=(2,2,3)$ \label{pA};
    \item nets of cubic curves in $\P^2$, which corresponds to $(k,d,n)=(2,3,2)$ \label{pB}; and
    \item plane curves of degree four, which corresponds to $(k,d,n)=(0,4,2)$ \label{pC}.
\end{enumerate}

A description of (semi)stability of plane quartics can be found in \cite[Section 1.12]{mumford} or in the proof of \cite[Proposition 6.3]{ascher2024wall}. A plane quartic is GIT stable if and only if it has at worst $\mathbf{A}_2$ singularities.
It is GIT semistable if and only if it is a reduced curve with at most double points that is not a cubic curve with an inflectional tangent line, or it is a double smooth conic.

\subsection{The log canonical threshold and toric valuations}\label{sec:lct}

As already mentioned, our strategy to address problems \hyperref[pA]{(A)}, \hyperref[pB]{(B)}, and \hyperref[pC]{(C)} above utilizes some of the ideas introduced in \cite{hz} about log canonical thresholds and toric valuations, as well as their connection to GIT stability, as seen also in \cite{compint}. We summarize these and refer the reader to \cite[Section 2.4]{hz} and \cite{komo} for more details.

We first define log pairs.

\begin{defi}
A \textbf{log pair} $(X,\Delta)$ consists of a normal quasi-projective variety $X$ together with an effective $\Q$-divisor $\Delta$ on $X$ such that $K_X+\Delta$ is $\Q$--Cartier.
\end{defi}

Now, given a log pair $(X,\Delta)$ and a prime divisor $E$ over $X$, i.e., a prime divisor on $Y$ with a projective birational morphism $\mu:Y\to X$ from a normal variety, one can consider the valuation $\ord_E$ on $X$ that sends each rational function in $K(X)^{\times} = K(Y)^{\times}$ to its vanishing order along $E\subset Y$. More generally, we can then consider (real) divisorial valuations.

\begin{defi}\label{def:valuation}
   A valuation $v$ of $X$ is \textbf{divisorial} if $v=c\cdot \ord_E$ for some real number $c>0$ and for some prime divisor $E$ over $X$.  
\end{defi}

This allows us to introduce then the notion of \textit{log discrepancy}. Given a log pair $(X,\Delta)$ and a divisorial valuation $v=\ord_E$ on $X$, where $E$ is a prime divisor on some normal variety $Y$ with a projective birational morphism $\mu:Y\to X$, the corresponding log discrepancy is the number 
\begin{align*}
A_{X,\Delta}(E) \coloneqq  1+\ord_E\!\bigl(K_Y-\mu^*(K_X+\Delta)\bigr).
\end{align*}
When $\Delta=0$, we simply write $A_X(E)$. Observe that the following identity $A_{X,\Delta}(E)=A_X(E)-\ord_E(\Delta)$ holds when $K_X$ and $\Delta$ are $\mathbb{Q}$-Cartier.

In turn, log discrepancies allow us to introduce certain types of singularities of pairs arising from the minimal model program, which are closely related to the notions of GIT (semi)stability in our setup and will be needed in Sections \ref{sec:git_cubics} and \ref{sec:new_results}.

\begin{defi}\label{defi--log--pair}
\leavevmode
\begin{enumerate}[(i)]
    \item  A log pair $(X,\Delta)$ is {\it log canonical} (lc) if $A_{X,\Delta}(E)\geq 0$ for all prime divisors $E$ over $X$. Similarly, we say that $(X,\Delta)$ is {\it Kawamata log terminal} (klt) if  $A_{X,\Delta}(E)> 0$ for any prime divisor $E$ over $X$. Furthermore, we say that $X$ has only {\it canonical singularities} if $\Delta=0$ and $A_{X}(E)\ge1$ for any prime divisor $E$ over $X$. 
      \item Moreover, if $E$ is a prime divisor on some normal variety $Y$ with a projective birational morphism $\mu\colon Y\to X$ such that $A_{X,\Delta}(E)=0$ (resp.~$A_{X,\Delta}(E)<0$), then we say that $c_X(E)$, which is the generic point of $\mu(E)$, is an {\it lc center} (resp.~{\it non-lc center}) of the log pair $(X,\Delta)$.
\end{enumerate}
\end{defi}

Let now $(X_1,\Delta_1)$ and $(X_2,\Delta_2)$ be two log pairs.
For any projective generically finite morphism $\mu_1\colon Y\to X_1$, we can define a $\Q$-divisor $\Delta^1_Y$ on $Y$ as follows.
Let $f_1\colon Z_1\to X_1$ be the Stein factorization of $\mu_1$ and let $\Delta_{Z_1}$ be such that $K_{Z_1}+\Delta_{Z_1}=f_1^*(K_{X_1}+\Delta_1)$.
Then set $\Delta_{Y}^1$ as the unique $\mathbb{Q}$-divisor such that $(g_1)_*\Delta_Y^1=\Delta_{Z_1}$ and $K_{Y}+\Delta_Y^1=g_1^*(K_{Z_1}+\Delta_{Z_1})$, where $g_1\colon Y\to Z_1$ is the canonical morphism.
Similarly, we can define $\Delta_{Y}^2$ as above for any projective generically finite dominant morphism $\mu_2\colon Y\to X_2$.

\begin{defi}\label{def:logcrep}
    Two log pairs $(X_1,\Delta_1)$ and $(X_2,\Delta_2)$ are said to be {\it log crepant} if there exist projective generically finite dominant morphisms $\mu_1\colon Y\to X_1$ and $\mu_2\colon Y\to X_2$ as above such that $\Delta_Y^1=\Delta_Y^2$. 
\end{defi}

We further note that, by \cite[Proposition 5.20]{komo}, if two log pairs $(X_1,\Delta_1)$ and $(X_2,\Delta_2)$ are log crepant, then $(X_1,\Delta_1)$ is lc (resp.~klt) if and only if so is $(X_2,\Delta_2)$.

In Sections \ref{sec:git_cubics} and \ref{sec:new_results}, we will also need the following numerical invariant.

\begin{defi}
    Assume now that $X$ has only lc singularities, that is, $(X,0)$ is lc, and let $D$ be an effective $\Q$-Cartier $\Q$-divisor on $X$. The \textit{log canonical threshold} (lct) of $(X,D)$ is the rational number
    $$\mathrm{lct}(X,D) = \sup\{\lambda \in \mathbb{Q}\,\mid\,(X,\lambda  D)\text{ is log canonical}\}.$$
\end{defi}

Alternatively, the log canonical threshold admits the following description in terms of divisorial valuations: 
\[
\mathrm{lct}(X,D)=\inf_E\frac{A_{X}(E)}{\mathrm{ord}_E(D)},
\]
where $E$ runs over all prime divisors over $X$ such that $\mathrm{ord}_E(D)\ne0$. In fact, the infimum is achieved by some prime divisor $E$ over $X$ by \cite[Corollary 2.31]{komo}, and we say that $E$ {\it computes} $\mathrm{lct}(X, D)$. Moreover, when $X$ and $Y$ admit a torus action and a projective birational morphism $\mu\colon Y\to X$ is equivariant with respect to the torus, torus-invariant prime divisors $E\subset Y$ play a special role, which motivates the following definition.

\begin{defi}
\label{DefToricDiv}
    Consider $X=\P^n$ as a toric variety with respect to some maximal torus $T\subset \SL(n+1)$ and let $\pi\colon Y\to X$ be a $T$-equivariant proper birational morphism from a normal variety. A prime divisor $E\subset Y$ is \textit{toric} if it is $T$-invariant. 
    Furthermore, if $v$ is a divisorial valuation such that $v=c\cdot \mathrm{ord}_E$ for some $c>0$ and some toric prime divisor $E$ over $X$, then we say that $v$ is a {\it toric valuation}.
\end{defi}

Note that in Definition \ref{DefToricDiv} when $n=2$, if $D_1$ and $D_2$ are toric prime divisors on $Y$  with respect to the same torus $T$ intersecting at a closed point $p$, then $p$ is $T$-invariant and hence the blowing up $\mu\colon Z\to Y$ at $p$ is also $T$-equivariant. In particular, the exceptional divisor $E$ of $\mu$ is toric. In addition, one can prove the following lemma.

\begin{lema}\label{lema--non-toric}
    Let $C\subset \mathbb{P}^2$ be a cubic curve and let $C_0\subset C$ be the largest effective divisor whose irreducible components are not toric. Let $p\in C$ be a closed point and $\pi_1\colon X_1\to\mathbb{P}^2$ be the blowing up at $p$. If $q$ is a closed point of the exceptional curve $\mathrm{Exc}(\pi_1)$  and $\pi_2\colon X_2\to\mathbb{P}^2$  is the blowing up at $q$, then $E=\mathrm{Exc}(\pi_2)$ is a toric divisor over $\P^2$ intersecting $(\pi_2\circ\pi_1)_*^{-1}C_0$ transversely or not intersecting $(\pi_2\circ\pi_1)_*^{-1}C_0$.
\end{lema}

\begin{proof}
    It follows from a careful but routine analysis of all possibilities for the plane cubic $C$ and the points $p$ and $q$, with $q$ infinitely near to $p$ as stated.
\end{proof}

The usefulness of Definition \ref{DefToricDiv} and the previous lemma to our purposes lies in the following criterion, which will be used in Section \ref{sec:git_cubics}.

\begin{lema}[{\cite[Corollary 2.8.1]{hz}}]\label{hypersurftoric}
A hypersurface  $H_f:=\{f=0\}\subset\mathbb{P}^n$  of degree $d$   is GIT unstable (resp. not stable) if and only if there exists a toric prime divisor $E$ such that \[
\frac{A_{\mathbb{P}^n}(E)}{\mathrm{ord}_E(H_f)}<\frac{n+1}{d} \qquad (\text{resp.}\, \leq).
\]
In particular, if a hypersurface  $H_f:=\{f=0\}\subset\mathbb{P}^n$  of degree $d$ is GIT unstable (resp. not stable)  for any $d\ge n+1$, then $\lct(X,H_f)<\tfrac{n+1}{d}$ (resp.~$\le$).
\end{lema}

For completeness, we also recall the notion of ADE singularities for plane curves, which play a central role in Section \ref{sec:construction_extended}. If $X$ is a smooth surface and $\Delta$ is a reduced curve on $X$, we say that $\Delta$ has only {\it ADE singularities} if either $\Delta$ is smooth, or at a singular point $\Delta$ is analytically locally described by one of the following equations.
    \begin{itemize}
        \item[$\mathbf{A}_n$:] $x^2+y^{n+1}, n\geq 1$
        \item[$\mathbf{D}_n$:]  $x^2y+y^{n-1}, n\geq 4$
        \item[$\mathbf{E}_6$:] $x^3+y^4$
        \item[$\mathbf{E}_7$:] $x^3+xy^3$
        \item[$\mathbf{E}_8$:] $x^3+y^5$
    \end{itemize}
 In particular, if $\Delta$ is two-divisible in $\mathrm{Pic}(X)$ and is reduced, then the normal surface obtained by taking a double cover of $X$ branched along $\Delta$ has only canonical singularities if and only if $\Delta$ has only ADE singularities.

\subsection{A cycle of correspondences and Gale duality}\label{sec:cycle}

Recall now that the Gale transform associates to a configuration of sufficiently general $s$ points in $\P^{r}$ a configuration of $s$ points in $\P^{s-r-2}$. Following the exposition in \cite{pointsets}, it can be described as follows. For a more modern exposition of many different aspects of the geometry of the Gale transform, we refer the reader to \cite{eisenbud_gale}.  

\begin{defi}
Let $\calP\in (\P^{r})^s$ be such that every sub configuration of $s-1$ points spans $\P^r$. A point $\mathcal{Q}\in (\P^{s-r-2})^{s}$ is called a Gale image of the point $\calP$ if there exists $D\in \mathrm{GL}(n)$ diagonal such that $A_{\calP}\cdot D\cdot (A_{\mathcal{Q}})^{t}=\mathbf{0}$, where $A_{\calP}\in \mathbb{M}_{(r+1) \times (s)}$ (resp. $A_{\mathcal{Q}}$) is a matrix whose columns consist of the projective coordinates of the $s$ points in $S_{\calP}$ (resp. $S_{\mathcal{Q}}$) and $\mathbf{0}$ denotes the $(r+1)\times(s-r-1)$ zero matrix.
\end{defi}

When $(s,r)=(7,3)$, the Gale transform
gives rise to a cycle of three beautiful correspondences among certain nets of quadrics in $\P^3$, nets of cubics in $\P^2$, and plane quartic curves, together with some extra data. This cycle is explained in detail,  for instance, in \cite[Chapter 6, Section 6.3.3]{cag},  \cite[Sections 6 and 7]{GrossHarris}, and \cite[Section 3]{veronese}. In a self-contained manner, it can be described as follows. 

First, observe that given $\N$ a net of quadric surfaces in $\P^3$, then the following lemma holds. A proof can be found in the references mentioned above.

\begin{lema}\label{lem:typical}
   The following statements about $\N$ are equivalent:
    \begin{enumerate}[(i)]
        \item the discriminant curve is smooth;
        \item  the base locus consists of eight distinct points, with no four coplanar;
        \item there are no quadrics of rank two or less in $\N$, and at a base point of the net, no quadric (in the net) is singular.
    \end{enumerate}
\end{lema}

As in \cite{GrossHarris}, let us call a net $\N\in \mathcal{X}_{2,2,3}$  \textit{typical} if it satisfies one of the conditions in Lemma \ref{lem:typical}. Then a typical net is determined by any seven of its base points. The projection from the eighth base point determines a point in $(\P^{2})^7$ as explained in \cite[p.~138]{GrossHarris}, which is the base locus of a net of plane cubics defining a (smooth) del Pezzo surface of degree two, hence a double cover of $\P^2$ branched along a quartic, which turns out to be isomorphic to the discriminant curve of the net of quadrics (cf., Proposition \ref{extended}). In other words, associated with this framework, there are three correspondences:

\begin{enumerate}
    \item[(C1)] \label{c1} typical nets of quadrics in $\P^3$ with a marked base point correspond to nets of plane cubics with a zero-dimensional base locus consisting of seven points in general position,
    \item[(C2)] \label{c2} the latter correspond to smooth plane quartics (together with a so-called Aronhold set) via the anti-canonical model of the associated del Pezzo surface, and
    \item[(C3)] \label{c3} smooth plane quartics marked with an Aronhold set correspond to typical nets since such a marking is equivalent to giving a representation of the curve as the determinant of a $4\times4$ symmetric matrix whose entries are linear forms in four variables. 
\end{enumerate}

\section{From nets of quadrics to nets of cubics and back} \label{sec:construction_extended}

\subsection{Generalized Gale transformation and good nets}
Let us now explain how one can extend the three correspondences \hyperref[c1]{(C1)}, \hyperref[c2]{(C2)}, and \hyperref[c3]{(C3)} from Section \ref{sec:cycle}, allowing for more degenerate unordered configurations of seven points in $\P^3$ (and in $\P^2$) that possibly contain infinitely near points. To this end, we first introduce the following definition.

\begin{defi}\label{def: regular}
A configuration $\calP\in \P^{3,8}$ (resp. $\calQ \in \P^{2,7}$) is called {\it regular} if condition (i) (resp. (ii)) below is satisfied.
\begin{enumerate}[(i)]
    \item  The linear system of quadrics $\N \coloneqq \calL_2(\calP)$ as in \hyperref[n2]{(N2)} is two-dimensional, with a zero-dimensional base locus (consisting of the eight points in $S_{\calP}$, see Proposition \ref{prop--ellip-fibration--criterion}) and one of the following equivalent conditions \cite[Assumption 1]{extremal} holds:
    \begin{enumerate}
        \item there exist two distinct members $Q_1$ and $Q_2$ of $\mathcal{N}$ such that $Q_1\cap Q_2$ is smooth at any base point of $\mathcal{N}$; 
        \item for any base point $p\in \P^3$ of $\mathcal{N}$ not infinitely near, there is at most one quadric $Q\in|\mathcal{N}|$ singular at $p$.
    \end{enumerate}
     \item  $S_{\calQ}$ is obtained from some regular point $\calP \in \P^{3,8}$ projecting seven points in $S_{\calP}$ from an eighth base point $p$ of $\calL_2(\calP)$ that is a base point in $\P^3$ not infinitely near. 
\end{enumerate}
Moreover, if (ii) holds, we abuse the terminology, and we further say that $\calQ$ is the {\it Gale image} of $\calP$ and write $\calQ=G(\calP,p)$ to emphasize that $\calQ$ strongly depends on $\calP$ and on the choice of the eighth $\P^3$-base point $p$.
  \end{defi}

In Definition \ref{def: regular} (ii) above, we mean that we consider the extended morphism on the blowing up of $\P^3$ at the chosen eighth base point, which is a $\P^3$-base point.

With Definition \ref{def: regular} in mind, we can now prove the following.

\begin{prop}\label{qtoc}
    Let $\N\in \mathcal{X}_{2,2,3}$ be a net of quadrics in $\P^3$ satisfying that there exists $\calP \in \P^{3,8}$ regular such that $\N=\calL_2(\calP)$. Choose $p$ as a $\P^3$-base point of $\N$ and write $\calQ=G(\calP,p)$ as in Definition \ref{def: regular}. Then the linear system of plane cubics $G(\N,p)\coloneqq \calL_3(\calQ)$ is two-dimensional. 
    \end{prop}

   \begin{proof}
Choose $\calP \in \P^{3,8}$ as in the statement and let $p\in \P^3$ be one of the base points of $\N=\calL_2(\calP)$. Take a general member $Q_1\in|\mathcal{N}|$ and choose $Q_2,Q_3$ such that $Q_1,Q_2$ and $Q_3$ generate $\mathcal{N}$.
    Let $g\colon \mathrm{Bl}_p(\mathbb{P}^3)\to \mathbb{P}^3$ be the blowing up at $p$ with exceptional divisor $E$ and let $q\colon \mathrm{Bl}_p(\mathbb{P}^3)\to \mathbb{P}^2$ be the projection of lines from $p$.
    Set $Q'_i:=g^*Q_i-E$ and $C_{ij}\coloneqq Q'_i\cap Q'_j$.

We first consider the case when $Q_1$, $Q_2$, and $Q_3$ are all general and smooth. Since $\mathcal{N}$ has no base locus of positive dimension, we can choose a smooth general quadric $Q_1\in|\mathcal{N}|$ such that $Q_1$ does not contain any line passing through $p$ or any other base point of $\mathcal{N}$, or the strict transform of a line in $\mathrm{Bl}_p(\mathbb{P}^3)$ passing through one of the base points infinitely near to $p$.
    Then, since $Q_1$ is smooth, we can view the strict transform $Q'_1$ as a one-point blowing up (at $p$) of $\mathbb{P}^1\times\mathbb{P}^1$. Let $L_1$ and $L_2\subset Q'_1$ be the strict transforms of the two lines in $\mathbb{P}^1\times\mathbb{P}^1$ passing through $p$.
    Let $\mathcal{N}'$ denote the net generated by $Q'_1,Q'_2$ and $Q'_3$ on $\mathrm{Bl}_p(\mathbb{P}^3)$.
    By the choice of $Q_1$, no base point of $\mathcal{N}'$ lies in $L_1$ or $L_2$.
    Furthermore, $\mathcal{N}'$ has seven base points $\{p_1,\ldots,p_7\}$ possibly containing infinitely near points, and these seven points are also the base points of the pencil generated by $C_{12}$ and $C_{13}$.
    Adding $p$ to $\{p_1,\ldots,p_7\}$, we can regard the eight base points $\calP$ of $\N$ as $\{p,p_1,\ldots,p_7\}$.
    We note that $C_{23}$ also passes through the base points of $\mathcal{N}'$ and that $q|_{Q'_1}\colon Q'_1\to\mathbb{P}^2$ is a birational morphism contracting only $L_1$ and $L_2$.
    Since the members of $\mathcal{N}'$ are ample and $\mathcal{N}'$ has no base point on $L_1$ or $L_2$, the support of $C_{12}$ or $C_{13}$ does not contain $L_1$ or $L_2$.
    Hence, $q_*C_{12}$ and $q_*C_{13}$ intersect at $q(p_1),\ldots,q(p_7),q(L_1)$ and $q(L_2)$.
    Replacing $Q_1$ with $Q_2$, we see that $q_*C_{23}$ is also a cubic and its intersection with $q_*C_{12}$ contains, at least, the points $q(p_1),\ldots,q(p_7)$.
    By the choice of $Q_1$, we see that $q_{*}C_{23}$ does not pass through $q(L_1)$ or $q(L_2)$. 
Hence, $q_*C_{12},q_*C_{23}$ and $q_*C_{31}$ generate a net of cubics and this net is contained in $G(\N,p)$ as its seven base points $q(p_1),\ldots,q(p_7)$ are precisely the points in $S_{\calQ}$, where $\calQ=G(\calP,p)$. In particular, $\mathrm{dim}\,G(\N,p)\ge2$.

We now observe that $\dim(G(\N,p))\leq 2$. Indeed, we can consider $Y$ the blowing up of $\mathbb{P}^2$ along the seven points $q(p_1),\ldots,q(p_7)$, which is a weak del Pezzo surface of degree two since $-K_Y$ is base point free (cf.~Definition \ref{defi--del-pezzo}). Since $\mathrm{dim}\,\H^0(Y,-K_Y)=3$ and $|G(\mathcal{N},p)|\subset |-K_Y|\subset |-K_{\mathbb{P}^2}|$, we obtain the desired upper bound.

Finally, we will show that $G(\N,p)$ is independent of the choice of generators $Q_1$, $Q_2$, and $Q_3$. 
First, we note that the construction in the second paragraph of this proof is independent of the choice of $Q_1$. 
Indeed, if you also choose $Q_2$ so general that $Q_2$ does not contain any line passing through two base points as in the previous paragraph, $\mathcal{Q}$ is also defined as $q(p_1),\ldots,q(p_7)$.
This means that if $Q_1$ and $Q_2$ are sufficiently general, $G(\calP,p)$ is the same when we replace the role of $Q_1$ and $Q_2$ in the previous discussion.
Next, let $f_1$, $f_2$ and $f_3\in V_{\N}$ be the polynomials corresponding to $Q_1$, $Q_2$ and $Q_3$.
For any $a\in \mathbb{C}$, let $Q_4$ be the quadric generated by $af_1+f_3$. Setting $Q'_i=g^*Q_i-E$ and $C_{ij}=Q_i'
\cap Q_j'$ as in the first paragraphs, we claim that $q_*C_{24}\in |G(\mathcal{N},p)|$.
Indeed, $C_{24}$ is a Cartier divisor on $Q'_2$ linearly equivalent to $C_{12}$ and to $C_{23}$ because if $Q'_2\cap Q'_4$ contains a subvariety of dimension two, then it contradicts the assumption that $\calP$ is regular.
Thus, $q_*C_{24}$ is linear equivalent to $q_*C_{12}$ and hence a cubic by \cite[Tag 02S1]{stacksproject}.
Moreover, letting $h_1$ and $h_3$ be the polynomials in $\H^0(\mathbb{P}^2,\mathcal{O}_{\mathbb{P}^2}(3))$ whose pullbacks to $Q'_2$ are $f_1|_{Q'_2}$ and $f_3|_{Q'_2}$ respectively, then the cubic $C'$ defined by $ah_1+h_3\in V_{G(\mathcal{N},p)}$ contains $q_*C_{24}$ scheme theoretically. Comparing their degrees, it follows that $|G(\mathcal{N},p)|\ni C'=q_*C_{24}$.  
This means that if you change the generators from $Q_1$, $Q_2$ and $Q_3$ to $Q_1$, $Q_2$, and $Q_4$, then $G(\N,p)$ is not changed. This shows that $G(\N,p)$ is independent of the choice of general generators as we claimed.
\end{proof}

\begin{rmk}
    In Proposition \ref{qtoc}, we can choose $p$ as a base point with multiplicity larger than one, and the isomorphism class of $G(\N,p)$ depends on the choice of $p$.
\end{rmk}

Proposition \ref{qtoc} above motivates the following definition.

\begin{defi}\label{def:good}
   A net $\N\in \mathcal{X}_{2,2,3}$ of quadrics in $\P^3$ is called \textbf{good} if there exists $\mathcal{P}\in \P^{3,8}$ regular such that $\N=\mathcal{L}_2(\calP)$. Similarly, a net $\mathcal{M}\in \mathcal{X}_{2,3,2}$ of cubics in $\P^2$ is called \textbf{good} if it arises from a good net of quadrics as in Proposition \ref{qtoc}, that is, $\mathcal{M}=\calL_3(\calQ)$ for some $\mathcal{Q}\in\P^{2,7}$ regular.
\end{defi}

Thus, we have established a correspondence between (what we call) good nets of quadrics in $\P^3$, together with a choice of a base point, and good nets of cubics in $\P^2$. We will refer to this correspondence $\N \mapsto G(\N,p)$ given by Proposition \ref{qtoc} as the {\it generalized Gale transformation} defined by the choice of the point $p$. This correspondence indeed generalizes the classical Gale transformation of typical nets (see \cite[p.~138]{GrossHarris}).  

The following examples show that there are indeed good nets of quadrics $\N$ with quite singular discriminant curves and that the seven points determined by a choice of the $\P^3$-base point of $\N$ can be quite degenerate.

\begin{exe}\label{exe-A_4}
    Choose coordinates $(x:y:z:w)$ in $\P^3$ and consider three quadric surfaces $Q_1$, $Q_2$ and $Q_3$ given by the equations $x^2-2xy=0$, $\tfrac14 x^2-xy+2xz+y^2+2yw=0$, and $2zw=0$. The net generated by $Q_1$, $Q_2$ and $Q_3$ has the associated discriminant quadric curve defined by the equation
\[
\begin{vmatrix}
\lambda+\tfrac14\mu & -\lambda-\tfrac12\mu & \mu & 0\\
-\lambda-\tfrac12\mu & \mu & 0 & \mu\\
\mu & 0 & 0 & \nu\\
0 & \mu & \nu & 0
\end{vmatrix}=\lambda^{2}\nu^{2}+2\lambda\mu^{2}\nu+\mu^{4}+\mu^{3}\nu=0.
\]
This is a quartic curve with an $\mathbf{A}_4$ singularity and an $\mathbf{A}_2$ singularity.
We can check that the net is good.
Moreover, the set-theoretic base locus of the net is the set $$\{(0\!:\!0\!:\!1\!:\!0),\ (0\!:\!0\!:\!0\!:\!1),\ (0\!:\!-2\!:\!0\!:\!1),\ (2\!:\!1\!:\!0\!:\!0)\},$$ with the four points having multiplicities three, two, one and two, respectively.
\end{exe}

\begin{exe}\label{exe-A_5}
Similarly, we consider the good net of quadrics generated by quadrics given by the equations $x^2-2xy=0$, $2xz+2yw=0$, and $\Bigl(\tfrac14-a_1\Bigr)x^2-xy+y^2+2zw=0$, for some $a_1\in \C$. Then the associated discriminant quadric curve has the equation
\[
\begin{vmatrix}
\lambda+\left(\tfrac14-a_1\right)\nu & -\lambda-\tfrac12\nu & \mu & 0\\
-\lambda-\tfrac12\nu & \nu & 0 & \mu\\
\mu & 0 & 0 & \nu\\
0 & \mu & \nu & 0
\end{vmatrix} =\lambda^{2}\nu^{2}+2\lambda\mu^{2}\nu+\mu^{4}+\mu^{2}\nu^{2}+a_1\nu^{4}=0
\]
and an $\mathbf{A}_5$ singularity. The base locus of the net contains $(0\!:\!0\!:\!1\!:\!0)$ with multiplicity four and $(0\!:\!0\!:\!0\!:\!1)$ with multiplicity two. If $a_1\neq 0$ it has two additional points $(2\!:\!1\!:\!t\!:\!-2t)$ with $t^2=-a_1$, each of multiplicity $1$. If $a_1=0$ (in which case there is also an $\mathbf{A}_1$ singularity as well as the $\mathbf{A}_5$ singularity), then there is only one additional point $(2\!:\!1\!:\!0\!:\!0)$ with multiplicity two.
\end{exe}

\begin{exe}\label{exe-A_6}
The good net generated by quadric surfaces given by equations $-2xy=0$, $-x^2+2xz+2yw=0$ and $y^2+2zw=0$ has the associated discriminant quartic curve with equation
\[
\begin{vmatrix}
-\mu & -\lambda & \mu & 0\\
-\lambda & \nu & 0 & \mu\\
\mu & 0 & 0 & \nu\\
0 & \mu & \nu & 0
\end{vmatrix} =\lambda^{2}\nu^{2}+2\lambda\mu^{2}\nu+\mu^{4}+\mu\nu^{3}=0,
\]
and hence contains an $\textbf{A}_6$ singularity. The set-theoretic base locus of this net is
$\{(0\!:\!0\!:\!1\!:\!0),\ (0\!:\!0\!:\!0\!:\!1),\ (2\!:\!0\!:\!1\!:\!0)\}$, and the three points have multiplicities three, four and one, respectively. 
\end{exe}

\begin{exe}\label{exe-E_7}
Finally, consider the three quadric surfaces given by the equations $z^2=0$,  $w(y+x)+xy=0$, and $zw+(x+y)^2=0$. Then the corresponding net is a good net of quadrics such that the associated discriminant quartic curve has equation
    \[
    16\begin{vmatrix}\nu & \tfrac{1}{2}\mu+\nu & 0 & \tfrac{1}{2}\mu \\
    \tfrac{1}{2}\mu+\nu & \nu & 0 & \tfrac{1}{2}\mu\\
    0 & 0 & \lambda & \tfrac{1}{2}\nu\\
    \tfrac{1}{2}\mu & \tfrac{1}{2}\mu & \tfrac{1}{2}\nu & 0
    \end{vmatrix}=\mu(4\lambda\mu^2+\nu^2(\mu+4\nu))=0,
    \]
   and hence an $\mathbf{E}_7$ singularity. Moreover, the set-theoretic base locus of the net in $\P^3$ consists of a single point with multiplicity eight, namely the point $(0:0:0:1)$ in the chosen coordinates. Note that this net defines an extremal rational elliptic threefold over $|\N|^\vee$ by \cite{extremal}.
\end{exe}

Next, we observe that any net of plane cubic curves with a regular base locus as in Definition \ref{def: regular} (ii)  gives rise to a double plane branched at a reduced plane quartic curve with at worst \textit{ADE} singularities (cf., Corollary \ref{cor:good-weak}). 
Before stating our result, we recall the definition of (weak) del Pezzo surfaces.

\begin{defi}\label{defi--del-pezzo}
 Let $Y$ be a smooth projective surface.
 If $-K_Y$ is big and nef, we say that $Y$ is a {\it weak del Pezzo surface}.
 The {\it degree} of $Y$ is set as the integer $K_Y^2>0$. If $-K_Y$ is ample, then we say that $Y$ is {\it del Pezzo}.
Otherwise, contracting all $(-2)$-curves $C$ such that $K_Y\cdot C=0$, we obtain a unique normal surface $\overline{Y}$ with only canonical singularities and with $-K_{\overline{Y}}$ ample.
We call $\overline{Y}$ the {\it anticanonical model} of $Y$.
\end{defi}

\begin{lema}\label{doubleplane}
    Let $\mathcal{M}$ be a good net of plane cubics. Let $\pi: Y\to \P^2$ denote the blowing up of $\mathbb{P}^2$ at the seven base points of $\mathcal{M}$, and let $\bar{\pi}:\bar Y \to \P^2$ denote the corresponding relative canonical model. Then $\overline{Y}$ is the anticanonical model of $Y$ and $\bar{\pi}$ is a degree two cover branched over a plane quartic $B$ with only ADE singularities.
\end{lema}

\begin{proof}
By assumption, the linear system $|-K_Y|$ on $Y$ is base-point free, and $-K_Y$ is big. Thus, $Y$ is a weak del Pezzo surface of degree two, and $\pi$ is a generically finite morphism of degree two (cf.~\cite{demazure}). Now, if $B$ denotes the branch divisor, then 
\[
K_{\bar Y}=\bar{\pi}^*\left(K_{\P^2}+\tfrac{1}{2}B\right)
\]
and therefore  $\left(K_{\P^2}+\tfrac{1}{2}B\right)^2=1$, which implies that $B$ has degree four.
Moreover, since $\overline{Y}$ has only canonical singularities, $B$ has only ADE singularities.
\end{proof}

\begin{rmk}\label{rmk:duality}
It is convenient to our purposes to diagram the maps in Lemma \ref{doubleplane} as follows.   
\begin{equation}\label{eq:diag_cubics}
\xymatrixcolsep{5pc} \xymatrix{
Y \ar[d]_{\pi} \ar[r]^{\text{resol. of sing.}} \ar[dr]^{\varphi}& \bar{Y} \ar[d]^{\bar{\pi}}_{2:1}\\
\P^2 \ar@{.>}[r]_-{\psi} & |\mathcal{M}|^{\vee}\simeq |-K_Y|^{\vee} \simeq \P^2 \supset B
}
\end{equation}
Moreover, it is also convenient to identify the $\P^2$ in which all the cubics in $\mathcal{M}$ live with the net itself in the following way. If $Z$ denotes the base locus of $\mathcal{M}$, then the rational map $\psi$, known as the Geiser involution, is such that given $p \in \P^2\backslash Z$ we have that $\psi(p)=\psi(p')$, where $p'$ is the ninth base point of the pencil of cubics through $Z\cup\{p\}$. Thus, each point in $q\in |\mathcal{M}|^{\vee}$ can be identified with a line in $\P^2$, namely the line joining the two points in $\psi^{-1}(q)$, hence with a point in $(\P^2)^{\vee}$.
\end{rmk}

In contrast, we have a result on the singularities of the discriminant of the net of quadric surfaces.
Before stating this, we recall the definition of the discriminant of a linear system of quadrics.

\begin{defi}
 Let $\calL\in\mathcal{X}_{k,2,n}$ be a $k$-dimensional linear system of quadrics in $\P^n$ with generators $Q_1,\ldots,Q_{k+1}$.   Suppose that each $Q_i$ is expressed as a symmetric matrix $A_i\in \mathbb{M}_{(n+1)\times(n+1)}$.
 Then $\calL$ defines a (possibly zero) homogeneous polynomial $\mathrm{det}\left(\sum_{i=1}^{k+1}\lambda_iA_i\right)$ of degree $(n+1)$ in $\mathbb{C}[\lambda_1,\ldots,\lambda_{k+1}]$.
 We set $\Delta(\mathcal{L})$ as the closed subscheme of $|\calL|\cong\P^k$ defined by $\mathrm{det}\left(\sum\lambda_iA_i\right)$.
 Here, if $\mathrm{det}\left(\sum\lambda_iA_i\right)\equiv0$, then $\Delta(\calL)=\P^k$ and we regard this as GIT unstable. 
\end{defi}

\begin{prop}\label{prop:good--ADE--corresp}
Let $\mathcal{N}\in \mathcal{X}_{2,2,3}$ be a net of quadrics in $\mathbb{P}^3$ and let $\Delta(\mathcal{N})$ be the discriminant. If $\Delta(\mathcal{N})$ has only ADE singularities, then the net $\N$ is good.
\end{prop}

\begin{proof}
We will prove that if $\N$ is not good, then  $\Delta(\N)=\P^2$, $\Delta(\mathcal{N})$ is not reduced, or it has singularities that are worse than ADE. In what follows, given a choice of three generators $Q_1, Q_2$ and $Q_3$ of $\N$ and a convenient choice of coordinates in $\P^3$, we will represent $\N$ by a net of $4\times4$ symmetric matrices $\lambda A_1+\mu A_2+\nu A_3$, with each matrix ${A}_i$ associated with $Q_i$ in the usual way.

First, assume that $\mathcal{N}$ has two members $Q_1$ and $Q_2$ and a base point $p$ such that $Q_1$ and $Q_2$ are singular at $p$. Then  $\Delta(\mathcal{N})$ must contain a double line or $\Delta(\N)=\P^2$. To see this, choose a third generator $Q_3$ and choose coordinates in $\P^3$ such that $p=(0:0:0:1)$ and $\N$ corresponds to a net of matrices as above such that $A_1$ is diagonal with zero for the fourth entry, $A_2$ has zero for all the entries of the fourth row and the fourth column and $A_3$ has zero for the entry at the fourth row and fourth column. Then, we see that $\mathrm{det}(\lambda A_1+\mu A_2+\nu A_3)$ is divisible by $\nu^2$ as an element of the polynomial ring $\mathbb{C}[\lambda,\mu,\nu]$.

Next, assume that $\mathcal{N}$ has a fixed part of dimension at least one. We may assume that $\Delta(\mathcal{N})\ne \P^2$ and then $\mathcal{N}$ has a smooth member. In this case, $\mathcal{N}$ has a curve $C\subset\mathbb{P}^3$ as a fixed part, and there are two cases to consider.
\begin{enumerate} 
    \item[{Case} I] Suppose that there exists a member $Q_1\in\mathcal{N}$ such that $Q_1$ is singular along $C$.
Then $\rank\,Q_1\le2$.
First, we deal with the case where $\rank\,Q_1=2$.
In this case, we may choose coordinates $(x:y:z:w)$ in $\P^3$ such that $Q_1=\{xy=0\}$ and $C=\{x=y=0\}$. 
Thus, letting $Q_2$ and $Q_3$ be two other generators of $\mathcal{N}$ and using the fixed coordinates to represent $\N$ by a net of $4\times4$ symmetric matrices as above, it follows that the discriminant curve $\Delta(\N)=\{\mathrm{det}(\lambda A_1+\mu A_2+\nu A_3)=0\}$  has a singular point at $(1:0:0)$ with multiplicity at least four. Indeed, since both $Q_2$ and $Q_3$ contain $C$, the two matrices $A_2$ and $A_3$ are of the form:

\begin{equation}\label{eq:matrix_quadruple_pt}    
    \begin{pmatrix}
           * & * & * & * \\ 
           * & * & * & * \\ 
           * & * & 0 & 0\\ 
           * & * & 0 & 0
        \end{pmatrix}.
\end{equation}

Next, suppose that $\rank\,Q_1=1$ and choose coordinates $(x:y:z:w)$ in $\P^3$ such that $Q_1=\{x^2=0\}$.
Let $Q_2$ and $Q_3$ be two other generators of $\N$.
In the fixed coordinates, we can represent $Q_2$ and $Q_3$ by matrices of the form 
\[
A_i=\begin{pNiceArray}{c|ccc}[margin]
 \ast & \ast & \ast & \ast \\ \hline
\ast & \Block{3-3}{\tilde{A}_i} & &  \\
\ast & & &  \\
\ast & &  & \\
\end{pNiceArray}, \qquad i=2,3,
\]
where $\tilde{A}_2$ and $\tilde{A}_3$ are $3\times 3$ (complex) symmetric matrices. Now, since $\mathcal{N}$ has the curve $C$ as a fixed part, we see that these two matrices $\tilde{A}_2$ and $\tilde{A}_3$ define two conics in the plane that share a common irreducible component. Thus, making a further suitable projective coordinate change, we may assume that
$Q_2$ and $Q_3$ are again represented by two matrices as in (\ref{eq:matrix_quadruple_pt}). In particular, in this case, the discriminant curve also has a singular point with multiplicity at least four.
\item[{Case} II] Suppose now that there is no member $Q\in\mathcal{N}$ such that $Q$ is singular along $C$.
Then for any point $p\in C$, there exists a quadric $Q_{p}\in \N$ that is singular at $p$ and $\Delta(\mathcal{N})$ has at least a double point at the point corresponding to $Q_p$. Since $Q_{p}$ and $Q_{p'}$ are different for any two general distinct points $p,p'\in C$, this implies that $\Delta(\mathcal{N})$ has infinitely many singularities, 
and we conclude that $\Delta(\mathcal{N})$ is not reduced.
\end{enumerate}
This completes the proof.
\end{proof}

In particular, given a good net $\N\in\mathcal{X}_{2,2,3}$, together with a choice of a base point $p\in \P^3$, and its associated good net of cubics $G(\N,p)$, one can now ask whether the corresponding plane quartics $\Delta(\N)$ and $B$ are isomorphic (cf. {\cite[Corollary 3.2]{eisenbud_gale} and \cite[Proposition 6.3.11]{cag}}).

\begin{prop}
Let $\N$ be a good net of quadrics in $\P^3$. For each choice of $p\in \P^3$, a base point of $\N$ not infinitely near, consider the weak del Pezzo surface $Y$ of degree two obtained by blowing up the seven points determined by $\calP$ and $p$. Then, the anticanonical model of $Y$ is isomorphic to the double cover of $\P^2$ branched along $\Delta(\N)$.  In particular, $\Delta(\N)$ has only ADE singularities.\label{extended}
\end{prop}

\begin{proof}
Set $\mathcal{M}:=G(\mathcal{N},p)$.
    With notations as in Lemma \ref{doubleplane} and as in the diagram (\ref{eq:diag_cubics}), what remains to be proven is the statement about the relative canonical model. As before, let us denote the branch curve of $Y$ by $B:=B(\mathcal{M})\subset |\mathcal{M}|^\vee$. We want to show that $B= \Delta(\N)$ via the canonical identification $|\mathcal{M}|^\vee\cong|\N|$ of two $\mathbb{P}^2$'s. 
    We first address the case of a general net of quadrics with eight distinct base points, that is, a net with a smooth discriminant quartic. 
    In this case, $\N$ is typical, and the assertion is well known; see, e.g., \cite[Proposition 6.3.11]{cag}. We give an alternative proof here for the reader's convenience.
    Since $\mathcal{N}$ is chosen to be general, we may assume that $-K_Y$ is ample.
    Then, since $B$ and $\Delta(\mathcal{N})$ are smooth quartics, it suffices to observe that $B\subset \Delta(\mathcal{N})$. If $q\in |\N|$ is an arbitrary point corresponding to some smooth member, that is, $q\notin\Delta(\N)$, then by choosing any other two generators of $\N$, the classical Gale transform gives a smooth subpencil of the net of cubics through the seven points whose generators meet transversely at two other points by the proof of Proposition \ref{qtoc}. This means that the natural finite morphism $Y\to\mathbb{P}^2$ is \'etale at $q$ and hence $q\not\in B$. This shows that $B\subset \Delta(\mathcal{N})$ as claimed.
    
Next, we will assume that $\N$ is a good net, not necessarily typical. Since $\mathcal{N}$ is a good net, we can take three general smooth generators $Q_1$, $Q_2$, and $Q_3$. Moreover, we can find a smooth curve $C$, together with a closed point $0\in C$, and $\mathcal{Q}_1$, $\mathcal{Q}_2$ and $\mathcal{Q}_3\in \pi_*\mathcal{O}_{\mathbb{P}^3_C}(2)$ satisfying the following three conditions, where $\pi\colon \mathbb{P}^3_C\to C$ is the canonical projection and $\P^3_C:=\mathbb{P}^3\times C$: 
\begin{enumerate}[(i)]
    \item the restriction of each $\mathcal{Q}_i \in \pi_*\mathcal{O}_{\mathbb{P}^3_C}(2)$ to $\mathbb{P}^3\times\{0\}$ coincides with $Q_i$, for $i=1,2,3$;
    \item the restrictions of $\mathcal{Q}_1$, $\mathcal{Q}_2$ and $\mathcal{Q}_3\in \pi_*\mathcal{O}_{\mathbb{P}^3_C}(2)$ to $\mathbb{P}^3\times\{t\}$ generate a good net $\mathcal{N}_t$ with a smooth discriminant plane quartic for any general $t\in C$; and
    \item  by replacing $C$ with its finite cover if necessary, there exists a section $\sigma\colon C\to\mathbb{P}^3_C$ defining a generalized Gale transform on each fiber as in Proposition \ref{qtoc}.
\end{enumerate}

Thus, by further shrinking $C$ and replacing it by a finite cover if necessary, we can construct $\mathcal{X}$ by blowing up $\mathcal{Q}_1$ along irreducible components of $\mathcal{Q}_1\cap\mathcal{Q}_2\cap\mathcal{Q}_3$ one by one.
We may assume that $\mathcal{X}$ is flat over $C$ since the irreducible components of $\mathcal{Q}_1\cap\mathcal{Q}_2\cap\mathcal{Q}_3$ are sections after replacing $C$ with its finite cover.
Hence, we can construct a flat family $\mathcal{B}$ of plane quartics in $\P^2_C$ defined by the family of weak del Pezzo surfaces. Considering the discriminant quartics of the nets $\N_t$ as $t$ varies, we also obtain a second flat family of plane quartics in $\P^2_C$, say $\tilde{\mathcal{B}}$. Since $\mathcal{B}_t=\tilde{\mathcal{B}}_t$ for any $t\in C\setminus\{0\}$, by the discussion in the first paragraph, it follows from \cite[III, Proposition 9.8]{Ha} that $B=\mathcal{B}_0=\tilde{\mathcal{B}}_0=\Delta(\N)$ as desired. The last statement follows from Lemma \ref{doubleplane}.
\end{proof}

\begin{rmk}
    Note that, by duality, a good net $\N$ of quadrics defines a net of elliptic quartics. The singular members of this dual net are thus parametrized by the dual of $\Delta$. At the same time, for each $p\in \P^3$ a base point of $\N$, $G(\N,p)$ is the net of plane cubics whose members are obtained by projecting these elliptic quartics from the point $p$ and, by Remark \ref{rmk:duality}, the dual of $B$ parametrizes its singular members. Therefore, it follows from the construction that, independent of $p$, the two quartics $\Delta(\N)$ and $B$ have the same dual curve.
    \end{rmk}

It is important to note that in Proposition \ref{extended}, the isomorphism class of the weak del Pezzo surface $Y$ is independent of the choice of $p$. In contrast, the isomorphism class of $G(\N,p)$ depends on the choice of $p$ because there are many choices for how to blow down $Y$ to $\P^2$ and $G(\calP,p)$ obviously depends on the choice of $p$. Hence, $Y$ does not determine the isomorphism class of $G(\N,p)$. We also observe that, by Propositions \ref{prop:main}, \ref{qtoc} and \ref{extended}, we have now completely succeeded in extending the threefold cycle of correspondences described in Section \ref{sec:cycle} to the case where the nets have a discriminant with (only) ADE singularities.

\begin{rmk}
    As a final remark, it is interesting to note that, in addition to the aforementioned threefold correspondence among nets of quadrics, nets of plane cubics, and plane quartics, these three objects also correspond to 28‑nodal double Veronese cones, as shown in \cite[Theorem 1.3]{veronese}. In particular, as in \cite[Question 10.6]{veronese}, one can ask whether this correspondence persists when either the double Veronese cone or the plane quartic acquires mild singularities, a question attributed to Shokurov and Prokhorov in \cite{veronese}. Since Proposition \ref{qtoc} extends the classical Gale correspondence to nets whose discriminant quartic has only ADE singularities, this suggests a natural compatibility with the framework considered in \cite{veronese}.
\end{rmk}

\subsection{The associated rational elliptic threefolds}\label{sec:rational_elliptic}

Finally, we conclude this section by establishing yet another criterion for a net $\N\in \mathcal{X}_{2,2,3}$ (resp. $\mathcal{M}\in \mathcal{X}_{2,3,2}$) to be good.
As shown in \cite{extremal}, blowing up the eight base points of a net of quadrics in $\P^3$ that is good as in Definition \ref{def:good} yields a rational elliptic threefold. Here, we prove that the converse statement also holds.

\begin{prop}\label{prop--ellip-fibration--criterion}
    Let $\mathcal{N}$ be a net of quadrics in $\mathbb{P}^3$ and let $f\colon X\to\mathbb{P}^2$ denote the canonical morphism obtained from resolving the indeterminacy of the rational map $\mathbb{P}^3\dashrightarrow \mathbb{P}^2\simeq |\N|^{\vee}$ induced by $\mathcal{N}$.
    Then $\mathcal{N}$ is good if and only if $\mathcal{N}$ has exactly eight base points $p_1,\ldots,p_8$ (possibly infinitely near), $X=\mathrm{Bl}_{p_1,\ldots,p_8}(\mathbb{P}^3)$  and $f\colon X\to\mathbb{P}^2$ gives $X$ the structure of a minimal elliptic fibration with a section.
\end{prop}

\begin{proof}
As shown in \cite{extremal} (see the argument just after \cite[Assumption 1]{extremal}), if $\mathcal{N}$ is good, then it defines an elliptic fibration structure $\mathrm{Bl}_{p_1,\ldots,p_8}\mathbb{P}^3\to |\N|^\vee$. Thus, we will only deal with the converse statement.

    First, observe by Bertini's theorem that given any general line $H\simeq \P^1\subset \mathbb{P}^2$, the restriction $f|_{f^{-1}(H)}\colon f^{-1}(H)\to H$ is a smooth rational elliptic surface with a section. Now,   
    under the identification $|\mathcal{N}|^\vee\simeq \mathbb{P}^2$ we may assume that each line $H$ corresponds to a quadric $Q\in|\mathcal{N}|$. In fact, if $\varphi\colon X\to \mathbb{P}^3$ denotes the canonical morphism, then $Q=\varphi_*(f^{-1}(H))$ and we see that for any general $H$, $Q$ is an irreducible quadric isomorphic to either $\mathbb{P}^1\times\mathbb{P}^1$ or $\mathbb{P}(1:1:2)$.
    We set $g\colon \tilde{Q}\to Q$ as the minimal resolution of singularities of $Q$.
    We note that if $Q=\mathbb{P}^1\times\mathbb{P}^1$, $g$ is the identity morphism.
    Otherwise, $\tilde{Q}=\mathbb{F}_2$ is the Hirzebruch surface of degree two. 
    In both cases, $\tilde{Q}$ has Picard number $\rho(\tilde{Q})=2$. Choose now two other generators $Q_1$ and $Q_2$ for $\mathcal{N}$. Then $g^*Q_1$ and $g^*Q_2$ generate a pencil $\Lambda$ whose movable part induces the elliptic fibration  $f|_{f^{-1}(H)}$. 
    Note that neither $Q_1$ nor $Q_2$ contains the support of $Q$.
    Let $q\colon f^{-1}(H)\to\tilde{Q}$ be the natural morphism induced by $\varphi$, since $f^{-1}(H)$ is a smooth rational elliptic surface with Picard number $\rho(f^{-1}(H))=10$, comparing Picard numbers, we see that $q$ is a blowing up at eight points.

    If $g$ is not an isomorphism, we claim that $\Lambda$ has no fixed part along the $(-2)$-curve. Otherwise, since $\N$ has exactly eight base points and the fixed part of $\Lambda$ is exceptional over $Q$, the minimal resolution of the base locus of $\Lambda$ is a blowing up at $n$-points with $n<8$ and hence not a smooth rational elliptic surface. Moreover, we claim that we may assume that $Q\cong \mathbb{P}^1\times\mathbb{P}^1$. Indeed, if $Q=\mathbb{P}(1:1:2)$ and $Q_1$ and $Q_2$ have a common zero at the singular point of $Q$, then $\Lambda$ would have a fixed part, necessarily along the $(-2)$ curve. Thus, $\mathcal{N}$ has no base point at the singular point of $Q$, and by Bertini's theorem, we may replace $Q$ with a smooth quadric.

Finally, we show that general members of $\Lambda$ are smooth.
For any general irreducible and reduced curve $C$ in $\Lambda$, we have that $p_a(C)=1$ since $C$ is a complete intersection of two quadrics. Now, we may assume that there exists a smooth fiber $\tilde{C}$ of $f|_{f^{-1}(H)}$ such that $q_*(\tilde{C})=C$.
Since $p_a(\tilde{C})=p_a(C)$, we have that $C$ is smooth.
In particular, if we pick $Q'\in |\mathcal{N}|$ such that $Q\cap Q'=C$, then we have that $Q'$ is also smooth locally around $C$ because $C$ is a smooth Cartier divisor of $Q'$.
Thus, condition (a) in Definition \ref{def: regular} is satisfied, and $\N$ is indeed a good net.    
\end{proof}

In addition, the argument in the proof of Proposition \ref{prop--ellip-fibration--criterion} gives the following corollary, which complements Lemma \ref{doubleplane}.

\begin{cor}\label{cor:good-weak}
Let $\mathcal{M}\in \mathcal{X}_{2,3,2}$ be a net of cubics in $\mathbb{P}^2$. Then $\mathcal{M}$ is good if and only if $\mathcal{M}$ defines a weak del Pezzo surface of degree two. 
\end{cor}

\begin{proof}
Taking Lemma \ref{doubleplane} into account, it suffices to show that if $\mathcal{M}$ defines a weak del Pezzo surface of degree two, then there exists a net $\N$ of quadrics in $\mathbb{P}^3$, together with a choice of a $\P^3$-base point $p$ of $\N$, such that $\mathcal{M}=G(\mathcal{N},p)$.   
To show this, take two general members $C_1$ and $C_2\in\mathcal{M}$ such that $C_1$ and $C_2$ intersect transversely at exactly two points $p_1$ and $p_2$ outside of the base locus of $\mathcal{M}$. Let $L$ be the line passing through $p_1$ and $p_2$ and let $r\colon S\to \mathbb{P}^2$ be the blowing up at $p_1$ and $p_2$. Here, we may assume that $C_1$ and $C_2$ are smooth. Then, each of the curves $r_*^{-1}C_1$ and $r_*^{-1}C_2$ intersects with $r_*^{-1}L$ transversely at one point. Let $q\colon S\to \mathbb{P}^1\times\mathbb{P}^1$ be the contraction of $r_*^{-1}L$ and let $\Lambda$ be the pencil generated by the smooth elliptic curves $q_*r_*^{-1}C_1$ and $q_*r_*^{-1}C_2$. By construction, $\Lambda$ has exactly eight base points, possibly including infinitely near points.
Consider the Segre embedding $\mathbb{P}^1\times \mathbb{P}^1\hookrightarrow \mathbb{P}^3$ and let $Q_1$ denote its image. Since the canonical morphism
\[
\H^0(\mathbb{P}^3,\mathcal{O}_{\mathbb{P}^3}(2))\to \H^0(Q_1,\mathcal{O}_{\mathbb{P}^1}(2)\boxtimes\mathcal{O}_{\mathbb{P}^1}(2))
\]
is surjective because $\H^1(\P^3,\mathcal{O}_{\P^3})=0$, there exist quadric surfaces $Q_2$ and $Q_3$ such that $Q_1\cap Q_2=q_*r_*^{-1}C_1$ and $Q_1\cap Q_3=q_*r_*^{-1}C_2$.
Furthermore, we see that the minimal resolution of the base points of the pencil generated by $r^{-1}_*C_1$ and $r^{-1}_*C_2$ of $S$ is a smooth rational elliptic surface, and $p_a(q_*r_*^{-1}C_1)=p_a(q_*r_*^{-1}C_2)=1$.
Thus, by the proof of Proposition \ref{prop--ellip-fibration--criterion}, we see that $q_*r_*^{-1}C_1$ and $q_*r_*^{-1}C_2$ are smooth elliptic curves, and $Q_1$, $Q_2$, and $Q_3$ generate a good net $\mathcal{N}$ of quadrics in $\mathbb{P}^3$. Now, it suffices to show that $\mathcal{L}=G(\mathcal{N},p)$ for some choice of $p$ a base point of $\N$ in $\P^3$.
We let $p\in Q_1$ be the image of $r_*^{-1}L$.
Then, by construction, $Q_1$ does not contain any line in $\mathbb{P}^3$ passing through $p$ or any other base point of $\mathcal{N}$, including points infinitely near to $p$.
Thus, we can define $G(\mathcal{N},p)$ as in the proof of Proposition \ref{qtoc} and we have that $G(\mathcal{N},p)=\mathcal{M}$.
\end{proof}

\section{GIT stability of linear systems of plane cubics}\label{sec:git_cubics}

In this section, we establish a stability criterion for linear systems of plane cubics in general and then apply it to nets, thereby providing a tool for proving our main result. We will use the background material introduced in Section \ref{sec:lct} throughout.

\begin{thm}\label{thm--cubics}
    Fix an integer $n\geq 2$ and let $C_1,\ldots,C_n$ be cubic curves in $\mathbb{P}^2$ different from each other. Then the pair $\left(\mathbb{P}^2,\tfrac{1}{n}(C_1+\ldots+C_n)\right)$ is lc if and only if the degree $3n$ curve $C_1+\ldots+C_n$ is GIT semistable with respect to the natural action of $
    \SL(3)$.
Furthermore, $C_1+\ldots+C_n$ is GIT stable if and only if the log pair $\left(\mathbb{P}^2,\tfrac{1}{n}(C_1+\ldots+C_n)\right)$ is klt or lc with only one lc center that is a smooth conic on $\P^2$.
\end{thm}

\begin{proof}
If the log pair  $\left(\mathbb{P}^2,\tfrac{1}{n}(C_1+\ldots+C_n)\right)$ is lc (resp.~klt or lc with only one lc center that is a smooth conic on $\P^2$), then the curve $C_1+\ldots+C_n$ is GIT semistable (resp.~stable) by Lemma \ref{hypersurftoric} because a smooth conic is not toric. Moreover, by Lemma \ref{hypersurftoric}, it suffices to show that if $(\mathbb{P}^2,\tfrac{1}{n}(C_1+\ldots+C_n))$ is not lc (resp. has a lc center that is not a smooth conic), then there exists a toric prime divisor $E$ over $\mathbb{P}^2$ such that 
    \begin{equation}
    A_{\mathbb{P}^2}(E)-\tfrac{1}{n}\mathrm{ord}_E(C_1+\ldots+C_n)<0 \quad(\mathrm{resp. }\le0).\label{eq--lct--toric}
    \end{equation}
    Assume that $(\mathbb{P}^2,\tfrac{1}{n}(C_1+\ldots+C_n))$ is not lc (resp.~has a lc center that is not a smooth conic) but no toric prime divisor over $X$ satisfies \eqref{eq--lct--toric}. 
    Then there exists a non-toric prime divisor over $\mathbb{P}^2$ that satisfies \eqref{eq--lct--toric}. 
    Let $E$ be a prime divisor over $\mathbb{P}^2$ that satisfies \eqref{eq--lct--toric}. In the latter case, we can choose $E$ not to be a smooth conic. Let $$X=X_l\to X_{l-1}\to \ldots \to X_1\to X_0=\mathbb{P}^2$$ be a succession of one-point blowups of the center of $E$.
    
   First, we consider the case where $l=0,1$, or $2$. If $l=0$ and $A_{\mathbb{P}^2}(E)-\tfrac{1}{n}\mathrm{ord}_E(C_1+\ldots+C_n)<0$, since the $C_i$'s are different from each other, then $\mathrm{deg}\,E\le1$ and hence $E$ is toric, which contradicts the assumption above.
    Consider the case where $l=0$ and $A_{\mathbb{P}^2}(E)-\tfrac{1}{n}\mathrm{ord}_E(C_1+\ldots+C_n)=0$, $\mathrm{deg}E\le 2$ by assumption. Since $E$ is not toric, we see that $E$ is a smooth conic, which also contradicts the choice of $E$. 
    If $l=1$ or $2$, then $E$ is toric (see the paragraph below Definition \ref{DefToricDiv}), and this is also a contradiction.
    
Thus, we may assume that $l\ge 3$.
For any $1\le i\le l$, let $E_i$ be the exceptional divisor of $\pi_i\colon X_i\to X_{i-1}$.
Here, we note that $E=E_l$.
Let $1\le j\le l$ be the smallest number such that $E_j$ is not toric.
By the above discussion, we see that $j\ge3$.
Let $L$ be the line passing through $c_{\P^2}(E_{j-1})$ such that $\pi_2$ is the blowing up at the common point of $E_1$ and $(\pi_1^{-1})_*L$. Let $f_i\colon X_i\to\P^2$ be the canonical morphism for any $i$. Note that $f_{j-1}^{-1}(c_{\P^2}(E_j))$ is a tree of rational curves and if $T\subset \SL(3)$ is a maximal torus making $E_{j-1}$ toric, then $f_{j-1}$ is $T$-equivariant. 
We may assume that
\[
A_{\mathbb{P}^2}(E_{j-1})-\tfrac{1}{n}\mathrm{ord}_{E_{j-1}}(C_1+\ldots+C_n)\ge0.
\]
Let $D\subset C_1+\ldots+C_n$ be the non-toric part.
By Lemma \ref{lema--non-toric} and the assumption that $C_i$ are cubics, $E_{j-1}$ and $(f_{j-1}^{-1})_*D$ have intersection multiplicity at most $n$ for any point and $j\ge3$.
Observing the process of blowing up in $f_{j-1}$, we see that $c_{X_{j-1}}(E_j)$ is not a nodal point of $F$ or not on $(f_{j-1}^{-1})_*L$ or the strict transform of $E_1$. Thus, we see that no toric divisor on $X_{j-1}$ passes through $c_{X_{j-1}}(E_j)$.
If $$A_{\mathbb{P}^2}(E_{j-1})-\tfrac{1}{n}\mathrm{ord}_{E_{j-1}}(C_1+\ldots+C_n)=0,$$
then $c_{X_{j-1}}(E_j)$ is never a non-lc center by \cite{kawakita2007inversion} applied to $(X_{j-1},E_{j-1}+\tfrac{1}{n}(f_{j-1}^{-1})_*D)$.
Here, note that $K_{X_{j-1}}+E_{j-1}+\tfrac{1}{n}(f_{j-1}^{-1})_*D-f_{j-1}^*(K_{\P^2}+\tfrac{1}{n}(C_1+\ldots+C_n))$ is effective locally around $c_{X_{j-1}}(E_j)$.
If $$A_{\mathbb{P}^2}(E_{j-1})-\tfrac{1}{n}\mathrm{ord}_{E_{j-1}}(C_1+\ldots+C_n)>0,$$
then $c_{X_{j-1}}(E_j)$ is never a lc center by a similar argument.
However, this contradicts the assumption that $E$ satisfies \eqref{eq--lct--toric} in both cases.
Therefore, we conclude the proof.
\end{proof}

We remark that if we choose $C_1=Q+L_1$ and $C_2=Q+L_2$ for some smooth conic $Q$ and general lines $L_1$ and $L_2$, then $(\P^2,\tfrac{1}{2}(C_1+C_2))$ is not klt but $C_1+C_2$ is GIT stable.

We deduce the following from Theorem \ref{thm--cubics}.

\begin{cor}\label{cor--cubic--linear--sys}
    Let $\calL\in \mathcal{X}_{k,3,2}$ be a $k$-dimensional linear system of plane cubic curves for any $k\ge1$.
    Then, $\calL$ is GIT (semi)stable if and only if for any choice of generators $C_1,\ldots, C_{k+1}\in\mathcal{L}$, the log pair $\left(\mathbb{P}^2, \tfrac{1}{k+1}(C_1+\ldots+C_{k+1})\right)$ is klt (resp.~lc).
\end{cor}

\begin{proof}
 Combining Theorem \ref{thm--cubics} with Theorem \ref{thm:stab_linear_systems} and Lemma \ref{hypersurftoric}, we see that it suffices to show that if for some choices of generators $C_1,\ldots, C_{k+1}\in\mathcal{L}$, the log pair $\left(\mathbb{P}^2, \tfrac{1}{k+1}(C_1+\ldots+C_{k+1})\right)$ is not klt but lc with only one lc center that is a smooth conic on $\P^2$, then for some other choices of generators $C'_1,\ldots, C'_{k+1}\in\mathcal{L}$, there exists a toric prime divisor $E$ over $\P^2$ such that 
 \begin{equation}
 A_{\mathbb{P}^2}(E)-\tfrac{1}{k+1}\mathrm{ord}_E(C'_1+\ldots+C'_{k+1})\le0.\label{eq--final--de--attekure}
 \end{equation}
 Let $Q$ be the smooth conic, which is the unique lc center of $\left(\mathbb{P}^2, \tfrac{1}{k+1}(C_1+\ldots+C_{k+1})\right)$.
 Then, $Q$ is the fixed part of $\calL$, $k\le2$ and we can find generators $C'_1,\ldots,C'_{k+1}$ for $\calL$ such that $C'_1-Q$ intersects $Q$ tangentially at a point $p$, and $C'_2-Q$ intersects with $Q$ at the same point $p$ if $k=2$.
 Let $\pi\colon\mathrm{Bl}_p(\P^2)\to\P^2$ be the blowing up at this point $p$ and let $q$ be the common point of $F:=\mathrm{Exc}(\pi)$ and $\pi_*^{-1}(Q)$.
 Since $A_{\P^2,\tfrac{1}{k+1}(C'_1+\ldots+C'_{k+1})}(F)=\tfrac{1}{k+1}$ and $\pi_*^{-1}(C'_1-Q)$ also passes through $q$, if we let $E$ be the exceptional divisor of the blowing up at $q$, then $E$ is toric and satisfies \eqref{eq--final--de--attekure}.
\end{proof}
 
 When $k=1$, this provides an alternative description of Miranda's result \cite{stabMiranda} on GIT stability of pencils of cubic curves. When $k=2$, it gives us a tool to prove Theorem \ref{thm:main} and also allows us to establish the following.

\begin{cor}\label{cor: unstable_subpencil}
    Let $\mathcal{M}$ be a net of plane cubic curves. 
    Suppose that $\mathcal{M}$ contains a smooth cubic curve.
    If $\mathcal{M}$ is not GIT stable, then $\mathcal{M}$ contains an unstable subpencil. Conversely, if $\mathcal{M}$ is good and contains an unstable subpencil corresponding to one of the cases (1), (2), (3), or (4) of \cite[Proposition 5.1]{stabMiranda}, then $\mathcal{M}$ is not stable.
    \end{cor}

\begin{proof}
Suppose that $\mathcal{M}$ is not GIT stable. By assumption and by \cite[Theorem 3.3]{hz}, we can choose generators $C_1, C_2$ and $C_3$ for $\mathcal{M}$ with $C_1$ smooth and such that $C_1+C_2+C_3$ is not GIT stable. Since $C_1$ is GIT stable, \cite[Lemma 2.3]{hz} implies that $C_2+C_3$ is unstable. By \cite[Theorem 3.3]{hz}, this means that the subpencil generated by $C_2$ and $C_3$ is unstable. Conversely, if $\mathcal{M}$ is good and contains an unstable subpencil that belongs to the list as in the statement, since $\mathcal{M}$ has a zero-dimensional base locus consisting of seven base points, then it is routine to check that there exist generators $C_1$, $C_2$, and $C_3$ for $\mathcal{M}$ such that $(\mathbb{P}^2,\tfrac{1}{3}(C_1+C_2+C_3))$ is not klt.
Thus, the assertion follows from Corollary \ref{cor--cubic--linear--sys}.
\end{proof}

\section{The proof of our main theorem}\label{sec:new_results}

We are finally in a position to prove Theorem \ref{thm:main}. Given any good net $\mathcal{M}\in \mathcal{X}_{2,3,2}$, consider the associated rational map $\psi\colon \mathbb{P}^2\dashrightarrow |\mathcal{M}|^{\vee}\simeq\mathbb{P}^2$ of degree two and the corresponding generically finite morphism of degree two $\varphi\colon Y\to\mathbb{P}^2$ induced by $\psi$ with branch locus a quartic curve $B\subset\mathbb{P}^2$ as in (\ref{eq:diag_cubics}). In this setup, we first prove the following.

\begin{prop}
\label{prop--log--crepant}
With notations as above, fix any three generators $C_1$, $C_2$ and $C_3$ of $\mathcal{M}$ and let $H_1$, $H_2$ and $H_3$ be the corresponding lines in $ |\mathcal{M}|^{\vee}\simeq \mathbb{P}^2$ (see Remark \ref{rmk:duality}).
    Let $E$ be the exceptional divisor of $\pi: Y\to\mathbb{P}^2$, the blowing up of $\mathbb{P}^2$ at the seven base points of $\mathcal{M}$, and let $\tilde{C}_i:=\pi^*C_i-E$.
    Then, the following three log pairs are pairwise log crepant in the sense of Definition \ref{defi--log--pair}: $\left(\mathbb{P}^2,\tfrac{1}{3}(C_1+C_2+C_3)\right)$,
        $\left(Y,\tfrac{1}{3}(\tilde{C}_1+\tilde{C}_2+\tilde{C}_3)\right)$, and 
        $\left(\mathbb{P}^2,\tfrac{1}{3}(H_1+H_2+H_3)+\tfrac{1}{2}B\right)$.
\end{prop}

\begin{proof}
    First, the assertion that the first two pairs are log-crepant follows from the same argument as in the proof of \cite[Lemma 5.4]{hz}.
    For the remaining log pairs, consider the Stein factorization $g\colon Y\to Y'$ of $\varphi\colon Y\to\mathbb{P}^2$, and let $\varphi'\colon Y'\to \mathbb{P}^2$ denote the canonical morphism.
     By \cite[Proposition 5.20]{komo} applied to $\varphi'$, $(Y',\tfrac{1}{3}g_*(\tilde C_1+ \tilde C_2+ \tilde C_3))$ and $(\mathbb{P}^2,\tfrac{1}{3}(H_1+H_2+H_3)+\tfrac{1}{2}B)$ are log crepant.
     In particular, we have that $K_{Y'}+\tfrac{1}{3}g_*(\tilde C_1+\tilde C_2+\tilde C_3)\sim_{\mathbb{Q}}0$.
     Now, since $K_{Y'}+\tfrac{1}{3}g_*(\tilde C_1+\tilde C_2+\tilde C_3)\sim_{\mathbb{Q}}0$ and $K_Y+\tfrac{1}{3}(\tilde C_1+\tilde C_2+\tilde C_3)\sim_{\mathbb{Q}}0$, we can apply \cite[Lemma 3.39]{komo} to deduce that 
     \[
     K_Y+\tfrac{1}{3}(\tilde C_1+\tilde C_2+\tilde C_3)=g^*\left(K_{Y'}+\tfrac{1}{3}g_*(\tilde C_1+\tilde C_2+\tilde C_3)\right).
     \]
     Thus, the second and third pairs are log crepant, which completes the proof.
     \end{proof}

Combining Corollary \ref{cor--cubic--linear--sys} and Proposition \ref{prop--log--crepant}, we thus obtain the following criterion relating GIT stability of $\mathcal{M}$ and that of $B$.

\begin{cor}\label{cor--quartic-to-net-cubics}
    With the same notations as in Proposition \ref{prop--log--crepant}, if $B$ is GIT (semi)stable, then so is the net of cubics $\mathcal{M}$.
    \end{cor}

\begin{proof}
By Corollary \ref{cor--cubic--linear--sys} and Proposition \ref{prop--log--crepant}, it suffices to show that if the branch curve $B$ is GIT (semi)stable, then for any choice of three linearly independent lines $H_1$, $H_2$ and $H_3$, the log pair $(\mathbb{P}^2,\tfrac{1}{2}B+\tfrac{1}{3}(H_1+H_2+H_3))$ is klt (resp.~lc). The argument is as follows. 

First, note that since $Y$ is smooth, $B$ is reduced and has at worst ADE singularities. Thus, by \cite[Theorems 1.3, 1.4]{ascher2024wall}, the plane quartic $B$ is GIT (semi)stable and not S-equivalent (cf.~Section \ref{sec:git}) to a double conic if and only if $(\mathbb{P}^2,\tfrac{3}{4}B)$ is klt (resp.~lc) (see also \cite[Proposition 6.3]{ascher2024wall}). In particular, if the curve $B$ is GIT (semi)stable and not S-equivalent to a double conic, then it is clear that $(\mathbb{P}^2,\tfrac{1}{3}(H_1+H_2+H_3)+\tfrac{1}{2}B)$ is klt (resp.~lc), since the pair $(\mathbb{P}^2,(H_1+H_2+H_3))$ is lc for any choice of three linearly independent lines $H_1$, $H_2$ and $H_3$. 

Next, we consider the case where $B$ is GIT semistable and S-equivalent to a double conic. Since $B$ is reduced, it has only $\mathbf{A}_i$ singularities (cf.~\cite[Section 1.12]{mumford}), and since $B$ is S-equivalent to a double conic, we see that $i\ge4$. 
Now, since $(\mathbb{P}^2,\tfrac{3}{4}B)$ is klt outside of the $\mathbf{A}_i$ singularities of $B$, so is the pair $(\mathbb{P}^2,\tfrac{1}{2}B+\tfrac{1}{3}(H_1+H_2+H_3))$ for any choice of three linearly independent lines $H_1$, $H_2$, and $H_3$.

Here, setting $\mathcal{D} \coloneqq \tfrac{1}{2}B+\tfrac{1}{3}(H_1+H_2+H_3)$, we further claim that $(\mathbb{P}^2, \mathcal{D})$ is lc for any choice of $H_1$, $H_2$, and $H_3$ but not klt for some choice of $H_1$, $H_2$, and $H_3$.
Let $p\in B$ be an $\mathbf{A}_i$ singularity for some $i\ge4$.
Since at least one of the lines $H_j$ does not pass through $p$, by changing the order, we may assume that $H_3$ does not pass through $p$. Let $\pi_1\colon \mathrm{Bl}_p\mathbb{P}^2\to \mathbb{P}^2$ be the blowing up at $p$ and $E_1$ the exceptional curve.
Let $q$ be the unique closed point contained in $(\pi_1)_*^{-1}B\cap E_1$ and $\pi_2\colon Z\to \mathrm{Bl}_p\mathbb{P}^2$ be the blowing up at $q$ with the exceptional curve $E_2$.
Since $H_1$ and $H_2$ are linearly independent, changing the ordering if necessary, we may assume that $(\pi_1)_*^{-1}H_2$ does not pass through $q$.
Since $B$ is a semistable plane quartic, $B$ does not consist of a cubic and an inflectional tangent line by \cite[Section 1.12]{mumford}. By this, we see that $H_1\not\subset B$ and hence $(\pi_2)_*^{-1}(\pi_1)_*^{-1}B$ and $(\pi_2)_*^{-1}(\pi_1)_*^{-1}H_1$ are disjoint around $E_2$. Thus, $0\le A_{\mathbb{P}^2,\mathcal{D}}(E_1)\le1$ and $0\le A_{\mathbb{P}^2,\mathcal{D}}(E_2)\le1$. Moreover, it suffices to consider the case where the lines $H_1$, $H_2$, and $H_3$ satisfy that $H_2$ passes through $p$ and $(\pi_1)_*^{-1}H_1$ passes through $q$. Indeed, in this case, $(\P^2,\mathcal{D})$ is the most singular at $p$ among all choices of $H_1$, $H_2$ and $H_3$.
Therefore, it suffices to consider this case, and we then have that $A_{\mathbb{P}^2,\mathcal{D}}(E_2)=0$.
Let $D_Z$ be the $\mathbb{Q}$-divisor determined by $(\pi_1\circ\pi_2)_*D_Z=\mathcal{D}$ and
\[
K_Z+D_Z=\pi_2^*\pi_1^*\left(K_{\mathbb{P}^2}+\mathcal{D}\right).
\]
Then $D_Z$ is effective and
it follows from \cite{kawakita2007inversion} that $(Z,D_Z)$ is strictly lc around $\pi_2^*E_1$. It is then routine to check that $(\P^2,\mathcal{D})$ is lc for any choice of $H_1$, $H_2$, and $H_3$, which completes the proof.
\end{proof}

We now relate the stability of a point $\N\in \mathcal{X}_{2,2,3}$ to the stability of the corresponding discriminant curve. We first prove the following.

\begin{prop}\label{prop: stab_quartic_net}
    Let $\N\in \mathcal{X}_{2,2,3}$ be a net of quadrics in $\mathbb{P}^3$ and denote its discriminant by $\Delta(\N)$. If $\mathcal{N}$ is GIT (semi)stable, then $\Delta(\N)$ is GIT (semi)stable. 
\end{prop}

\begin{proof}
We will prove the contrapositive statement. Let $Q_1$, $Q_2$, and $Q_3$ denote the three quadrics in $\mathcal{N}$ corresponding to the points $(0:0:1)$, $(1:0:0)$ and $(0:1:0)$, respectively. Without loss of generality, we may assume that the point $(0:0:1)$ lies in $\Delta(N)$ and $Q_1$ is singular, and thus that the rank of $Q_1$ is at most three.

We first exhibit an equation for $\Delta(\N)$ under the assumption that $\Delta(\N)$ is not stable. Recall that $\Delta(\mathcal{N})$ is not GIT stable if and only if $\Delta(\mathcal{N})$ is not reduced or it has a singularity worse than $\mathbf{A}_2$ or at least a triple singularity (cf.~\cite[Section 1.12]{mumford}).
Up to a change of the projective coordinates $(\lambda:\mu:\nu)$ of $\mathbb{P}^2$, we may assume that the polynomial defining $\Delta(\mathcal{N})$ is expressed as 
\begin{equation}\label{eq--delta(N)}
a_0\lambda^2\nu^2+a_1\lambda\mu^3+a_2\mu^4+a_3\lambda\mu^2\nu+a_4\lambda^2\mu^2+a_5\lambda^3\nu+a_6\lambda^4+a_7\lambda^3\mu+a_8\lambda^2\mu\nu.
\end{equation}
We note that the one-parameter subgroup $\rho:=\mathrm{Diag}(1,0,-1)$ degenerates \eqref{eq--delta(N)} to $a_0\lambda^2\nu^2+a_2\mu^4+a_3\lambda\mu^2\nu$ with the Hilbert--Mumford weight zero if $a_0a_2a_3\ne0$.
In this case, the GIT semistability of $\Delta(\N)$ is equivalent to that of the curve defined by $a_0\lambda^2\nu^2+a_2\mu^4+a_3\lambda\mu^2\nu$ by \cite[Propositions 2.2 and 2.3]{GIT}.
If $a_0=a_2=a_3=0$, then $\mu(\Delta(\N),\rho)<0$ or $\Delta(\N)=\P^2$, and hence $\Delta(\N)$ is GIT unstable.
Therefore, $\Delta(\N)$ is GIT unstable if and only if $a_0=0$ or $a_2=a_3=0$ in the expression of \eqref{eq--delta(N)}.

Then consider the three possible ranks of $Q_1$ separately. If $\rank Q_1=1$, the following argument shows that $\mathcal{N}$ is unstable. We may choose coordinates $(x:y:z:w)$ in $\P^3$ such that $Q_1$ is given by $x^2=0$. Then it is easy to check that the one-parameter subgroup $\rho_1:=\mathrm{Diag}(3,{-1},{-1},{-1})$ satisfies that $\mu(\N,\rho_1)<0$ since $\mu(\N,\rho_1)\le \mu(Q_1,\rho)+\mu(Q_2,\rho)+\mu(Q_3,\rho)$ (cf.~\cite[Proposition 2.1 and Lemma 3.1]{hz}). Moreover, in this case, it is routine to check that $\Delta(\mathcal{N})$ has a singularity with multiplicity at least three and hence is unstable. 
Therefore, it suffices to deal with the cases where $2\le\rank Q_1\le3$.

 If $\rank Q_1=3$, we can assume that the point $(0:0:0:1)$ is the unique singularity of $Q_1$ and, with notations as in the proof of Proposition \ref{prop:good--ADE--corresp}, that  

\begin{align*}
       A_1=\begin{pmatrix}
           * & * & * & 0 \\ 
           * & * & 0 & 0 \\ 
           * & 0 & 0 & 0 \\  
           0 & 0 & 0 & 0
        \end{pmatrix},
        A_2=\begin{pmatrix}
           * & * & * & * \\ 
           * & * & * & * \\ 
           * & * & * & 0\\  
           * & * & 0 & 0
        \end{pmatrix},
        A_3=\begin{pmatrix}
           * & * & * & * \\ 
           * & * & * & 0 \\ 
           * & * & 0 & 0 \\  
           * & 0 & 0 & 0
        \end{pmatrix},
        \end{align*}
observing equation \eqref{eq--delta(N)}. Then the one-parameter subgroup $\rho_2:=\mathrm{Diag}(3,1,{-1},{-3})$ has $\mu(\N,\rho_2)\leq 0$. Similarly, if $\rank\,Q_1=2$, then we may assume that 
\begin{equation}A_1=
\begin{pmatrix}
           0 & 1 & 0 & 0 \\ 
           1 & 0 & 0 & 0 \\ 
           0 & 0 & 0 & 0 \\  
           0 & 0 & 0 & 0
\end{pmatrix}.\label{eq-a_1}
\end{equation}          
Then, since the coefficients of $\mu\nu^3$, $\mu^2\nu^2$, and $\mu^3\nu$ in the equation \eqref{eq--delta(N)} of $\Delta(\mathcal{N})$ are zero, we may assume that
\begin{equation}
A_3=\begin{pmatrix}
           * & * & * & a \\ 
           * & * & * & b \\ 
           * & * & c & 0 \\  
           a & b & 0 & 0
\end{pmatrix},\label{eq-a_3}
\end{equation}  
with $abc=0$. In particular, up to a further change of coordinates, we may assume that either $b=0$ or $c=0$.
If $c=0$, we see that the one parameter subgroup $\rho_3:=\mathrm{Diag}(1,1,{-1},{-1})$ has $\mu(\N,\rho_3)\leq 0$. Furthermore, if $c\ne0$ and $b=0$, since the coefficient of $\lambda\mu\nu^2$ in the equation of $\Delta(\mathcal{N})$ vanishes, then the point $(0:0:0:1)$ lies in $Q_2$. In this case, the one-parameter subgroup $\rho_4:=\mathrm{Diag}(1,0,0,{-1})$ satisfies that $\mu(\N,\rho_4)\le0$.

Finally, suppose that $\Delta(\mathcal{N})$ is unstable. In this case, we have that $a_2=a_3=0$ or $a_0=0$ in \eqref{eq--delta(N)} as already observed, and we claim that $\N$ is unstable.
Since we have already considered the case where $\rank Q_1=1$, two possibilities remain. 
If $\rank Q_1=2$, we use \eqref{eq-a_1} and \eqref{eq-a_3}. First, consider the case where $c=0$ and $a_0=0$. If $\mu(\N,\rho_3)= 0$, we see that $\rho_3$ degenerates $Q_2$ to a double hyperplane. Thus, $\rho_3$ degenerates $\mathcal{N}$ to an unstable net by the discussion above. Therefore, $\N$ is unstable in this case because the GIT stability of $\N$ and that of this net are equivalent by \cite[Propositions 2.2 and 2.3]{GIT} in this case.
Next, consider the case where $a_2=a_3=c=0$.
After a suitable change of coordinates of $\P^3$, we may assume that 
\begin{equation*}
A_3=\begin{pmatrix}
           * & * & a'& 0 \\ 
           * & * & b' & 0 \\ 
           a' & b' & 0 & 0 \\  
           0 & 0 & 0 & 0
\end{pmatrix},
A_2=\begin{pmatrix}
           * & * & *& * \\ 
           * & * & * & * \\ 
           * & * & * & * \\  
           * & * & * & d
\end{pmatrix}
\end{equation*} 
with $a'b'd=0$.
If $d=0$, then $\mu(\N,\rho_5)<0$, where $\rho_5:=\mathrm{Diag}(1,1,1,{-3})$.
Otherwise, we may assume that $b'=0$ and replace $\N$ with the degeneration of $\N$ by $\rho_3$. Then $\mu(\N,\rho_4)=0$ and $\rho_4$ degenerates $Q_2$ to a double hyperplane.
Thus, $\N$ is unstable again.
Next, consider the case where $c\ne0$ and $a_0=0$.
Then, we see that
\begin{equation*}
A_2=\begin{pmatrix}
           * & * & *& * \\ 
           * & * & * & * \\ 
           * & * & * & 0 \\  
           * & * & 0 & 0
\end{pmatrix}.
\end{equation*}
Setting $\rho_6:=\mathrm{Diag}(1,1,0,-2)$, $\mu(\N,\rho_6)\le0$ and if $\mu(\N,\rho_6)=0$, then $\N$ degenerates to a net generated by quadrics defined by equations $xy$, $xw$ and $yw$, respectively, where we use the canonical projective coordinate $(x:y:z:w)$ of $\P^3$.
This net is clearly unstable, and therefore $\N$ is unstable.
Finally, we deal with the case where $a_2=a_3=0$ but $c\ne0$.
In this case,
\begin{equation*}
A_3=\begin{pmatrix}
           * & * & *& a \\ 
           * & 0 & 0 & 0 \\ 
           * & 0 & c & 0 \\  
           a & 0 & 0 & 0
\end{pmatrix},
A_2=\begin{pmatrix}
           * & * & *& * \\ 
           * & * & * & 0 \\ 
           * & * & * & * \\  
           * & 0 & * & 0
\end{pmatrix}.
\end{equation*}
If $\mu(\N,\rho_4)=0$, then $\rho_4$ degenerates $\N$ to a net generated by three quadrics defined by equations $xy$, $2axw+cz^2$ and $zw$, respectively.
Since $\mathrm{Diag}(1,-3,1,1)$ has the negative Hilbert-Mumford weight of this net, $\N$ is unstable.
If $\rank Q_1=3$ and $\mu(\N,\rho_2)=0$, we can see that the degeneration $\N'$ of $\N$ by $\rho_2$ is unstable as follows. 
Let $Q'_i\in |\N'|$ be the degeneration of $Q_i$ for $i=1,2,3$. 
Note that $Q_1=Q_1'$.
When $a_0=0$, we see that $\rank Q'_3=1$ and $\N'$ is unstable.
When $a_2=a_3=0$, $Q'_1=\{\alpha xz+\beta y^2=0\}$, $Q'_2=\{xw=0\}$ and $Q'_3=\{yw=0\}$ for some $\alpha,\beta\in\mathbb{C}$ unless $\rank Q'_3=1$. It is easy to check that $\N'$ is unstable considering $\mathrm{Diag}(0,-1,-2,3)$ in this case, and therefore $\N$ is unstable by \cite[Propositions 2.2 and 2.3]{GIT}.
In any case, the conclusion is that $\N$ is unstable as claimed.
\end{proof}

\begin{rmk}\label{rmk-lz}
    Yuchen Liu and Junyan Zhao informed us that they have also independently obtained this result \cite{YuchenJunyan}. 
\end{rmk}

We will now show that the converse statement also holds if one restricts to good nets of quadrics $\N$ as in Definition \ref{def:good}. This will follow from combining Proposition \ref{extended}, Corollary \ref{cor--quartic-to-net-cubics}, and the following result, which relates the stability of $\N$ to that of $G(\N,p)$.

\begin{prop}\label{prop:cubics-to-quadrics:git}
Let $\mathcal{N}$ be a good net of quadrics in $\mathbb{P}^3$ and let $Q_1$ be a general element of $\mathcal{N}$ with a base point $p\in\mathcal{N}$ in $\P^3$.
Suppose that $G(\mathcal{N},p)$ is GIT stable (resp.~semistable).
Then for any choice of $Q_2$ and $Q_3$ such that $Q_1$, $Q_2$, and $Q_3$ generate $\mathcal{N}$, the pair $(\mathbb{P}^3,\tfrac{2}{3}(Q_1+Q_2+Q_3))$ is klt (resp.~lc). 
\end{prop}

\begin{proof}
     Let $g\colon \mathrm{Bl}_p(\mathbb{P}^3)\to \mathbb{P}^3$ be the blowing up at $p$ with the exceptional divisor $E$ and $q\colon \mathrm{Bl}_p(\mathbb{P}^3)\to \mathbb{P}^2$ the projection of lines from $p$.
    Set $Q'_i:=g^*Q_i-E$ and $C_{ij}=Q'_i\cap Q'_j$.
    By the proof of Proposition \ref{qtoc}, we see that $C_{12}$, $C_{23}$, and $C_{31}$ generate $G(\mathcal{N},p)$.
    Suppose that $G(\mathcal{N},p)$ is GIT stable (resp.~semistable).
    By Corollary \ref{cor--cubic--linear--sys}, then the log pair $(\mathbb{P}^2,\tfrac{1}{3}(C_{12}+C_{23}+C_{31}))$ is klt (resp.~lc). 

We note that all fibers of $q$ are isomorphic to $\mathbb{P}^1$.
Moreover, $Q_1'$, $Q_2'$, and $Q_3'$ generate a net with no fixed part.
This shows that $q^{-1}(x)\cap Q'_1$, $q^{-1}(x)\cap Q'_2$, and $q^{-1}(x)\cap Q'_3$ are all distinct points for any general closed point $x\in\mathbb{P}^2$. 
Therefore, any general fiber of $q$ is klt and isomorphic to each other as log pairs.
This means that $(\mathrm{Bl}_p(\mathbb{P}^3),\tfrac{2}{3}(Q'_1+Q'_2+Q'_3))\to \mathbb{P}^2$ is a lc-trivial fibration in the sense of \cite[Definition 2.1]{ambro2004shokurov} with every two general fibers isomorphic to each other. This shows that the moduli $\mathbb{Q}$-{\it b}-divisor on $\P^2$ in the sense of \cite[Definition 2.3]{ambro2004shokurov} is $\mathbb{Q}$-linearly trivial by \cite[Lemma 5.2 and Theorem 4.4]{ambro2004shokurov} and \cite[Propositions 3.1 and 4.4]{ambro2005moduli}.
 Set $B$ as an effective $\mathbb{Q}$-divisor on $\P^2$ such that the multiplicity of $B$ along $F$ is
\[
1-\sup\left\{a\in\mathbb{Q}\mid\left(\mathrm{Bl}_p(\mathbb{P}^3),\tfrac{2}{3}\left(Q'_1+Q'_2+Q'_3\right)+aq^*F\right)\textrm{ is lc around $q^{-1}(\eta_F)$} \right\}
\]
for any prime divisor $F$ on $\mathbb{P}^2$ with the generic point $\eta_F$. Note that $B$ is nothing but the discriminant $\mathbb{Q}$-$b$-divisor in the sense of \cite[Definition 2.3]{ambro2004shokurov}.
 Then $B\sim_{\mathbb{Q}}-K_{\P^2}$, and  $(\mathrm{Bl}_p(\mathbb{P}^3),\tfrac{2}{3}(Q'_1+Q'_2+Q'_3))$ is klt (resp.~lc) if and only if so is $(\mathbb{P}^2,B)$ by \cite[Theorem 3.1]{ambro2004shokurov}.

We claim that $B-\tfrac{1}{3}(C_{12}+C_{23}+C_{31})$ is effective.
Since $B$ is effective, it suffices to check the multiplicity of $B$ along the irreducible components of $C_{12}+C_{23}+C_{31}$.
Let $F$ be an irreducible component of $C_{12}+C_{23}+C_{31}$ with $\mathrm{ord}_F(C_{12}+C_{23}+C_{31})=m$.
We note that $1\le m\le 3$ since $(\mathbb{P}^2,\tfrac{1}{3}(C_{12}+C_{23}+C_{31}))$ is lc.
Without loss of generality, we may assume that $F\subset C_{23}$. Then the reduced structure of $Q_2'\cap Q_3'\cap q^*F$ is generically a section over $F$ not contained in $Q'_1$ since $\N$ is good. We may also assume that $q^*F\not\subset  Q'_2$. Then $F\not\subset C_{12}$, and $F\subset C_{31}$ if and only if $q^*F\subset Q'_3$.
In the case where $m=1$, $Q'_2$ and $Q'_3$ intersect transversely around the generic fiber of $q^{-1}(F)$.
In this case, it is easy to see that $\mathrm{ord}_F(B)\ge\tfrac{1}{3}$.
Similarly, it is easy to see that $\mathrm{ord}_F(B)\ge\tfrac{m}{3}$ for the cases where $m=2$ and $m=3$ if $F\not\subset C_{31}$. If $F\subset C_{31}$, then we see that $Q'_1\cap q^*F$ and $Q'_2\cap q^*F$ are generically sections over $F$ different from each other and hence $m=2$ and $\mathrm{ord}_F(B)\ge\tfrac{2}{3}$.
Therefore, $B-\tfrac{1}{3}(C_{12}+C_{23}+C_{31})$ is effective.

Since $\tfrac{1}{3}(C_{12}+C_{23}+C_{31})\sim_{\mathbb{Q}}-K_{\mathbb{P}^2}$, we see that $B=\tfrac{1}{3}(C_{12}+C_{23}+C_{31})$.
Therefore,   $\left(\mathrm{Bl}_p(\mathbb{P}^3),\tfrac{2}{3}\left(Q'_1+Q'_2+Q'_3\right)\right)$ is klt (resp.~lc).
Since this log pair is log crepant to $\left(\mathbb{P}^3,\tfrac{2}{3}(Q_1+Q_2+Q_3)\right)$, we obtain the assertion.
\end{proof}

As yet another consequence, we obtain the following complete criterion for GIT stability of good nets of quadric surfaces, thus completing the proof of Theorem \ref{thm:main}.

\begin{thm}\label{thm--equiv.good.nets}
    Let $\mathcal{N}$ be a net of quadrics in $\mathbb{P}^3$ such that $\Delta(\mathcal{N})$ is a reduced curve with only ADE singularities.
    Then, the following four statements are equivalent:
    \begin{enumerate}
        \item $\mathcal{N}$ is GIT stable (resp.~GIT semistable);
        \item $\Delta(\mathcal{N})$ is GIT stable (resp.~GIT semistable); 
        \item for some (and hence any) choice of a base point $p\in \P^3$ of $\mathcal{N}$, the net of cubics $G(\mathcal{N},p)$ is GIT stable (resp.~GIT semistable); and
    \item if $Q_1$ is a general element of $|\mathcal{N}|$, then for any choice of two other generators $Q_2$ and $Q_3\in |\mathcal{N}|$, the pair  $(\mathbb{P}^3,\tfrac{2}{3}(Q_1+Q_2+Q_3))$ is klt (resp.~lc).
    \end{enumerate} 
\end{thm}

\begin{proof}
By assumption, note that the net $\mathcal{N}$ is good by Proposition \ref{prop:good--ADE--corresp}. Now, the implication $(4)\Rightarrow (1)$ follows from \cite[Corollary 4.1]{hz}. By Proposition \ref{extended}, the implications $(1)\Rightarrow (2)$ and $(2)\Rightarrow (3)$ follow from Proposition \ref{prop: stab_quartic_net} and Corollary \ref{cor--quartic-to-net-cubics}, respectively.  Finally, the implication $(3)\Rightarrow (4)$ follows from Proposition \ref{prop:cubics-to-quadrics:git}.
      \end{proof}

More generally, we further deduce the following.

    \begin{cor}\label{cor--final}
Let $\mathcal{N}$ be a net of quadrics in $\mathbb{P}^3$.
Then $\mathcal{N}$ is GIT stable (resp.~semistable) if and only if $\Delta(\mathcal{N})$ is also GIT stable (resp.~semistable). 
    \end{cor}

    \begin{proof}
For the semistability, the assertion follows from Proposition \ref{prop: stab_quartic_net} and Corollary \ref{cor: unstable nets discriminants}.
For the strict stability, by Proposition \ref{prop: stab_quartic_net}, it suffices to show that if $\Delta(\N)$ is stable, then $\N$ is stable.  
    By \cite[Proposition 6.3]{ascher2024wall}, $\Delta(\N)$ has only ADE singularities.
    Thus, the assertion follows from Theorem \ref{thm--equiv.good.nets}. 
    \end{proof}

\begin{rmk}
    We note that by Theorem \ref{thm--equiv.good.nets} there are unstable (resp.~non-stable) good nets of quadrics in $\P^3$ as illustrated by Example \ref{exe-E_7} (resp.~Examples \ref{exe-A_4}, \ref{exe-A_5} and \ref{exe-A_6}). By Corollary \ref{cor--final} and the description of the GIT stability of plane quartics (cf.~\cite[Section 1.12]{mumford}), the latter examples are S-equivalent to nets with a discriminant that is a double smooth conic. We also note that a result similar to Corollary \ref{cor--final} holds for nets of conics in $\P^2$ by \cite{wall}, but we cannot always extend it to the higher-dimensional case by \cite[Remarks 3.20 and 3.21]{fedorchuk2013stability}.    
\end{rmk}

\section{Further questions}\label{sec:questions}

We record here a few directions for future research that naturally emerge from our project. In this paper, we provide precise links among three distinct GIT problems by considering a threefold cycle of correspondences arising from a generalization of Gale duality. A first natural question is whether our setting extends to higher dimensions and to other types of linear systems of hypersurfaces. 

\begin{que}\label{quest 1}
    More precisely, are there other relevant geometric examples where the GIT (semi)stability of a linear system of hypersurfaces through a configuration of $s$ points in $\P^r$ coincides with the GIT (semi)stability of a linear system of hypersurfaces through a configuration of $s$ Gale dual points in $\P^{s-r-2}$? Is there also a “discriminant” side? 
\end{que}

Positive answers to Question \ref{quest 1} would naturally lead to considering new explicit moduli problems involving different types of higher-dimensional fibrations and to comparing these moduli spaces.

In a related direction, it would be interesting to extend our findings to the context of K-stability, since K-stability has been highly successful in constructing moduli spaces for Fano varieties (see \cite{Xu2025} for details) and GIT is closely related to K-stability. As shown in \cite{zhou2024log,liu2024non}, the GIT moduli for divisors in a fixed K-polystable Fano variety can be regarded as a K-moduli space of certain log Fano pairs. In addition, explicit descriptions of K-moduli spaces are largely obtained by studying and classifying GIT quotients and identifying them with K-moduli spaces. For example, it is observed in \cite{mabuchi_mukai_1990, SSY, Pap22, MGPZ2024} that GIT moduli or VGIT moduli of nets of quadrics give exact descriptions of K-moduli spaces of specific log Fano pairs. There are also K-moduli wall crossing phenomena, which have been successfully explored as a comparison method between K-moduli and GIT moduli for log Fano pairs in \cite{ascher2024wall,liu2024non}. Despite these developments, a general framework that connects VGIT for linear systems of divisors to K-moduli has not yet been formulated. In line with this, we pose the following question.

\begin{que}
To what extent can our main theorem be generalized to natural wall crossing phenomena in VGIT or K-moduli? 
Can we relate the GIT moduli of linear systems to other significant moduli spaces?
\end{que}

Finally, we observe that good nets of quadric surfaces are also in correspondence with certain rational elliptic threefolds obtained as an eightfold blowing up of $\P^3$ (see Section \ref{sec:rational_elliptic}). We therefore find it appealing to ask the following.

\begin{que}\label{quest 4}
    Can we relate our stability criteria to intrinsic geometric data associated with these elliptic fibrations, such as the types of singular fibers similar to \cite{stabMiranda}? If so, which boundary strata would correspond to the extremal rational elliptic threefolds described in \cite{extremal}?
\end{que}

We hope that the geometric correspondences uncovered in this paper
will serve as a starting point for further interactions between GIT stability, birational geometry, wall crossing, and the explicit geometry of rational elliptic threefolds.

\appendix

\section{\texorpdfstring{GIT unstability of nets of quadrics in $\P^3$}{GIT unstability of nets of quadrics in P3} }\label{section--Theo's-calculation}

In this appendix, we extend the original work of C.T.C Wall in \cite{wall_theta} and provide a more ``geometric'' classification of unstable orbits for the GIT problem associated with the action of $ \SL(4)$ on $\mathcal{X}_{2,2,3}$ (see Section \ref{3git} for the setup). We achieve this using the algorithm from \cite{Pap22} and the computational code \cite{Pap_code}. We summarize below how this process works. A detailed explanation is given in \cite[\S 3.5]{Pap22}.

\subsection{An overview of the algorithm}
In this subsection, we explain the strategy to detect unstable nets of quadrics in $\P^3$ first described in \cite{Pap22}.  Fix a maximal torus $T$ in $\SL(4)$. Then, by \cite[Lemma 3.16]{Pap22}, we can find a finite set $P_{3,2,2}$ (which we will describe explicitly in Lemma \ref{one-ps for nets} below) of one-parameter subgroups of $T$ as \cite[Definition 3.15]{Pap22} that allows us to determine all unstable nets of quadrics with respect to the action of $\SL(4)$. More precisely, for any unstable net $\N$, we can find $g\in \SL(4)$ and $\lambda\in P_{3,2,2}$ such that $\mu(g\cdot \N,\lambda)<0$. 

From now on, we explain the fundamental notations in \cite[Section 3]{Pap22}. Let $\Xi_2$ be the set of monomials of degree two in $\mathbb{C}[x_0,x_1,x_2,x_3]$, which is the projective coordinate ring of $\P^3$ and $T$ acts on $\mathbb{C}[x_0,x_1,x_2,x_3]$ preserving the vector spaces $\mathbb{C}\cdot x_i$ for $i=0,1,2,3$.
We set $$\mathrm{Supp}(f):=\left\{x^I\in \Xi_2\middle|\text{if $f=\sum_{x^K\in \Xi_2}c_Kx^K$, then $c_I\ne0$}\right\}$$ for any homogeneous polynomial $f$ of degree two. 
For each $\lambda \in P_{3,2,2}$, we can set the {\it $\lambda$-order} $\le_{\lambda}$ on $\Xi_2$ as \cite[Section 3.2]{Pap22}.
For each $\lambda \in P_{3,2,2}$ and $x^I$ and $x^J\in \Xi_2$, we can define {\it maximal sets} $N^{-}(\lambda,x^I,x^J)$ (cf.~\cite[Definition 3.25 and Lemma 3.26]{Pap22}) as
\[
N^{-}(\lambda,x^I,x^J):=A\times B\times C\subset (\Xi_2)^3,
\]
where $A:=\{x^K\in\Xi_2|\mu(x^K,\lambda)+\mu(x^I,\lambda)+\mu(x^J,\lambda)<0\}$, $B:=\{x^K\in\Xi_2|x^K\le_{\lambda}x^I\}$, and $C:=\{x^K\in\Xi_2|x^K\le_{\lambda}x^I\}$. Here, we say that $x^I$ and $x^J$ are {\it supporting monomials}.
Note that $N^{-}(\lambda,x^I,x^J)$ are maximal in the sense that for any subsets $A',B'$ and $C'\in \Xi_2$ such that $A\subset A'$, $B\subset B'$ and $C\subset C'$, if $\mu(x^{K_1},\lambda)+\mu(x^{K_2},\lambda)+\mu(x^{K_3},\lambda)<0$ for any $x^{K_1}\in A'$, $x^{K_2}\in B'$ and $x^{K_3}\in C'$, then $A=A'$, $B=B'$ and $C=C'$. In particular, a net of quadrics $\mathcal{N}$ is unstable if and only if up to the action of an element $g\in \SL(4)$, the monomials of the three (fixed) generators of $g\cdot \mathcal{N}$ are contained in a triple $N^{-}(\lambda,x^I,x^J)$ for some $\lambda\in P_{3,2,2}$ and $x^I,x^J\in\Xi_2$ (cf.~\cite[Theorem 3.27]{Pap22}). More precisely, each of the sets $N^{-}(\lambda,x^I,x^J) = A\times B\times C$ defines three generators (up to the action of $\SL(4)$) of a GIT unstable net $\mathcal{N}$ of quadrics  with 
$$V_\mathcal{N} = \{c_1f_1+c_2f_2+c_3f_3\,\mid\,(c_1,c_2,c_3)\in \mathbb{C}^3\},$$
where $f_{1}$, $f_2$ and $f_3$ are homogeneous polynomials of degree two such that $\mathrm{Supp}(f_1)\subset A$, $\mathrm{Supp}(f_2)\subset B$ and $\mathrm{Supp}(f_3)\subset C$. Furthermore, if $\N$ is generated by $f_1$, $f_2$ and $f_3$ such that $\mathrm{Supp}(f_1)=A$, $\mathrm{Supp}(f_2)=B$ and $\mathrm{Supp}(f_3)=C$, then we say that the net $\mathcal{N}$ is \emph{maximally unstable} associated with $N^{-}(\lambda,x^I,x^J)$. 
In particular, classifying these maximal unstable orbits also allows us to determine the semistable orbits of the GIT quotient. In summarizing this, we obtain the following. 

\begin{thm}[{\cite[Theorem 3.17]{Pap22}}]\label{thm: pap paper thm}
    Let $\N$ be a net of quadrics in $\P^3$. Then $\N$ is unstable if and only if there exist generators $Q_1$, $Q_2$, and $Q_3$ for $\N$, $g \in \SL(4)$, $x^I$, $x^J\in \Xi_2$, and $\lambda \in P_{3,2,2}$ such that $g\cdot Q_1$, $g\cdot Q_2$, and $g\cdot Q_3$ are expressed by homogeneous polynomials $f_1$, $f_2$, and $f_3$ in $\mathbb{C}[x_0,x_1,x_2,x_3]$ of degree two with  $\mathrm{Supp}(f_1)\times \mathrm{Supp}(f_2)\times \mathrm{Supp}(f_3)\subset N^{-}(\lambda,x^I,x^J)$.
\end{thm}

The destabilizing one-parameter subgroups are explicitly described by computer software \cite{Pap_code} using the description in \cite[Definition 3.15]{Pap22}. For convenience, given a one-parameter subgroup $\lambda=\operatorname{Diag}(r_0,\dots,r_3)$ of $T$, we set $\overline{\lambda}=\operatorname{Diag}(-r_3,\dots,-r_0)$. If you take $P_{3,2,2}$ as \cite[Definition 3.15]{Pap22}, then it consists of $293$ elements, which we will not list here. By listing all $N^-(\lambda, x^I,x^J)$,  we see that many of the $293$ elements produce the same $N^-(\lambda, x^I,x^J)$ or that are contained in other maximal sets and can replace $P_{3,2,2}$ with the following set.

\begin{lema}\label{one-ps for nets}
    Any element of $P_{3,2,2}$ is $\lambda_k$ or $\overline{\lambda}_k$, where $\lambda_k$ is one of the following:
    \begin{equation*}
        \begin{split}
            \lambda_1 = \operatorname{Diag}(21, 17, 5, -43), & \quad \lambda_2 = \operatorname{Diag}(9, 1, -3, -7)\\
            \lambda_3 = \operatorname{Diag}(13, -3, -3, -7), & \quad \lambda_4 = \operatorname{Diag}(5,4,-3,-6)\\
            \lambda_5 = \operatorname{Diag}(29,21,-11,-39), & \quad \lambda_6 = \operatorname{Diag}(25,9,-15,-19)\\
            \lambda_7 = \operatorname{Diag}(31,19,-9,-41), & \quad \lambda_8 = \operatorname{Diag}(4, 3, 1, -8)\\
            \lambda_9 = \operatorname{Diag}(5, 1, -3, -3). &\quad 
        \end{split}
    \end{equation*}
\end{lema}

\subsection{Stability criteria and the discriminant quartic}
We are now in a position to describe the stability criteria. We first recall some results on the classification of pencils of quadrics.

Let $\calP$ be a pencil of quadrics in $\P^3$ generated by $Q_1$ and $Q_2$ that contains at least one smooth element. 
Assume that $Q_2$ is smooth.
Fix $A_1$ and $A_2\in\mathbb{M}_{4\times4}$ as symmetric matrices that express $Q_1$ and $Q_2$, respectively. 
Consider the determinant polynomial $\det(A_1+\lambda A_2)$ and let $\alpha_i\in\mathbb{C}$ be a root of $\det(A_1+\lambda A_2)$ of multiplicity $e_i$. 
Now we set $2\le h_i\le 4$ as the number such that $\alpha_i$ is also a root for every subdeterminant of $A_1+\lambda A_2$ whose size is at least $5-h_i$.

We then define $l^i_j$ to be the minimum multiplicity of a root $\alpha_i$ for the set of subdeterminants of size $5-j$, for $j =1,\dots ,h_i-1$. We can check that $l^i_j \geq l^{i}_{j+1}$ and define $e^i_j:= l^i_j - l^{i}_{j+1}$. Thus, we obtain
$$\det(A_1+\lambda A_2)=\prod_{i=1}^r\prod_{j = 1}^{h_i-1} (\lambda - \alpha_i)^{e^i_j},$$
where $\alpha_i$'s are roots different from each other.
Then, the \emph{Segre symbol} of the pencil $\mathcal{P}$ is written as $[(e^1_1,\dots e^1_{h_1-1}),\dots,(e^r_1,\dots e^r_{h_r-1}) ].$
The Segre symbol depends only on the isomorphic class of $\calP$ and is independent of the choice of $Q_1$ and $Q_2$.

Segre symbols classify pencils of quadrics up to projective equivalence by encoding the degenerations of the members of the pencil  (cf.~\cite{Hodge_book} and \cite[\S 4.2, Table 2]{Pap22}). For example, a pencil with the Segre symbol $[(3,1)]$ is generated by the quadrics $Q_1 = \{q_1(x_1,x_2,x_3) +x_0x_3=0\}$ and $Q_2 = \{x_3l(x_1,x_2,x_3)=0\}$ and the curve $Q_1\cap Q_2$ has a $\mathbf{D}_4$ singularity. We will use Segre symbols to classify all unstable orbits of the nets of quadrics in $\P^3$.

\begin{thm}\label{unstable nets}
    Let $\mathcal{N} $ be a net of quadrics in $\mathbb{P}^3$ generated by $Q_1$, $Q_2$ and $Q_3$. Then, $\mathcal{N}$ is GIT unstable if and only if the generators $Q_i$ are given as below or as a degeneration of the below (up to a $\SL(4)$-action):
    \begin{enumerate}
    \item $Q_1$ and $Q_2$ are smooth quadrics, and $Q_3$ is a double hyperplane, i.e. the intersection $Q_1\cap Q_2$ is smooth and $Q_i\cap Q_3$ is a double conic;
    \item $Q_1$ is smooth and $Q_2 = H_1\cup H_2$ and $Q_3 = H_1\cup H_3$ are the union of two hyperplanes, such that $Q_1\cap Q_i$ is a curve with a $\mathbf{D}_4$ singularity and $Q_2\cap Q_3 = H_1\cup (H_2\cap H_3)$; 
    \item $Q_1$, $Q_2$ are smooth and $Q_1\cap Q_2$ is the intersection with a $\mathbf{A}_1$ singularity, while $Q_i\cap Q_3$ is the intersection with a $\mathbf{D}_4$ singularity;
    \item $Q_1$ is smooth and $Q_2$, $Q_3$ are cones sharing the same singular point, such that the singular point is contained in $Q_1$;
    \item $Q_1$ is smooth and $Q_2 = H_1\cup H_2$ and $Q_3 = H_3\cup H_4$ are the union of hyperplanes such that $Q_1\cap Q_i$ is a curve with two $\mathbf{A}_1$ singularities;
    \item $Q_1$ is smooth, $Q_2 = H_1\cup H_2$ is a union of two hyperplanes and $Q_3$ is a singular quadric such that $Q_1\cap Q_2$ is a curve with two $\mathbf{A}_1$ singularities and $Q_1\cap Q_3$ is smooth;
    \item $Q_1$ is smooth, $Q_2$ and $Q_3$ are singular quadrics such that $Q_1\cap Q_i$ is  a curve with two $\mathbf{A}_1$ singularities;
    \item $Q_1$ is smooth, $Q_2$ and $Q_3$ are singular quadrics such that $Q_1\cap Q_2$ is a  curve with two $\mathbf{A}_1$ singularities and $Q_1\cap Q_3$ is smooth;
    \item $Q_1$, $Q_2$ are smooth, and $Q_3=H_1\cup H_2$ is a singular quadric such that $Q_2\cap Q_3$ is a double line and two lines in general position and $Q_1\cap Q_3$ is a curve with an $\mathbf{A}_3$ singularity;
    \item $Q_1$ is smooth, $Q_2$ and $Q_3$ are singular quadrics such that $Q_1\cap Q_i$ are curves with $2$ $\mathbf{A}_1$ singularities;
    \item $Q_1$ is smooth and $Q_2 = H_1\cup H_2$ and $Q_3 = H_1\cup H_3$ are the union of two hyperplanes, such that $Q_1\cap Q_i$ is a curve with two $\mathbf{A}_1$ singularities; 
    \item $Q_1$ and $Q_2$ are smooth and $Q_3 = H_1\cup H_2$ is the union of two hyperplanes, such that $Q_1\cap Q_2$ is a smooth curve and $Q_i\cap Q_3$ are reduced curves with an $\mathbf{A}_3$ singularities.
    \end{enumerate}
\end{thm}
\begin{proof}
Due to Theorem \ref{thm: pap paper thm}, a net $\mathcal{N}$ is unstable if and only if for some $g \in \operatorname{SL}(4)$, $\lambda\in P_{3,2,2}$, $x^I,x^J\in\Xi_2$, and three generators $Q_1$, $Q_2$ and $Q_3$ of $(g\cdot \mathcal{N})$, $\mathrm{Supp}(Q_1)\times\mathrm{Supp}(Q_2)\times \mathrm{Supp}(Q_3)\subset  N^{-}(\lambda, x^I, x^J)$.
Therefore, it suffices to classify all maximal $N^{-}(\lambda, x^I, x^J)$ up to orders of products $A\times B\times C=N^{-}(\lambda, x^I, x^J)$ and extract geometric information of maximally unstable nets.
These maximal sets can be found algorithmically using the computational package \cite{Pap_code}.
Furthermore, to conclude that maximally unstable nets belong to one of the lists, we will use the full classification of intersections of two quadrics in $\mathbb{P}^3$ presented in \cite[\S 4, Table 2]{Pap22} using Segre symbols. 

Let $(\lambda, x^I, x^J) = (\lambda_1, x_0^2,x_0^2)$. Then a general maximally unstable net $\N$ associated with the maximal set $N^{-}(\lambda, x^I, x^J)$ has three quadrics $Q_1$, $Q_2$ and $Q_3$ as generators such that $Q_1$, $Q_2$ are given by polynomials $f_1(x_0,x_1,x_2,x_3)$ and $f_2(x_0,x_1,x_2,x_3)\in V_{\N}$ such that $\mathrm{Supp}(f_1)=\mathrm{Supp}(f_2)=\Xi_2$. Then, $Q_1$ and $Q_2$ are smooth quadrics, and $Q_1\cap Q_2$ is also smooth by \cite[\S 4.2, Table 2]{Pap22}. Furthermore, $Q_3$ is given by the polynomial $f_2^3 = x_3^2$. In particular $Q_3$ is a double hyperplane. 
Let $\langle Q_i,Q_j\rangle$ denote the subpencil generated by $Q_i$ and $Q_j$.
Notice that the pencils $\langle Q_1,Q_3\rangle$ and $\langle Q_2,Q_3\rangle$ both have Segre symbols $[(1, 1, 1), 1]$. 
In particular, $Q_i\cap Q_3$ is a double conic for $i=1$, $2$ by \cite[\S 4.2, Table 2]{Pap22} (see also \cite{Hodge_book}). This case corresponds to (1).

        Similarly, let $(\lambda, x^I, x^J) = (\lambda_1, x_0x_3,x_0^2)$. Then a general maximally unstable net $\N$ associated with the maximal set $N^{-}(\lambda, x^I, x^J)$ has three quadrics $Q_1$, $Q_2$ and $Q_3$ as generators such that  $Q_1$ is smooth and $Q_2 = H_1\cup H_2$ and $Q_3 = H_1\cup H_3$ are the union of two hyperplanes. The classification of intersections follows from \cite[\S 4, Table 2]{Pap22} as in the last paragraph. This case corresponds to (2).

        Now, let $(\lambda, x^I, x^J) = (\lambda_2, x_1x_3,x_0x_3)$. Then a general maximally unstable net $\N$ associated with the maximal set $N^{-}(\lambda, x^I, x^J)$ has three quadrics $Q_1$, $Q_2$ and $Q_3$ as generators such that $Q_1$, $Q_2$ are smooth and $Q_1\cap Q_2$ is a reduced curve with a $\mathbf{A}_1$ singularity, while $Q_i\cap Q_3$ is a reduced curve with a $\mathbf{D}_4$ singularity. This case corresponds to (3).

        Letting $(\lambda, x^I, x^J) = (\lambda_3, x_1^2,x_0x_1)$, a general maximally unstable net $\N$ associated with the maximal set $N^{-}(\lambda, x^I, x^J)$ has three quadrics $Q_1$, $Q_2$ and $Q_3$ as generators such that    $Q_1$ is smooth, and $Q_2$ and $Q_3$ are cones of a rational curve of degree two. This case corresponds to (4).
        
Let $(\lambda, x^I, x^J) = (\lambda_4, x_2^2,x_0^2)$. Then a general maximally unstable net $\N$ associated with the maximal set $N^{-}(\lambda, x^I, x^J)$ has three quadrics $Q_1$, $Q_2$ and $Q_3$ as generators such that $Q_1$, is given by polynomial $f_1(x_0,x_1,x_2,x_3)\in V_{\N}$ such that $\mathrm{Supp}(f_1) = \Xi_2$.

Then $Q_1$ is a smooth quadric. In addition $Q_2$ and $Q_3$ are given by the polynomials $f_2(x_2,x_3)$ and $f_3(x_2,x_3)$, and thus are the unions of two hyperplanes. 
Notice that the pencils $\langle Q_1,Q_3\rangle$ and $\langle Q_1,Q_2\rangle$ both have the same Segre symbols $[(1, 1), 1, 1]$.  In particular, $Q_1\cap Q_i$ is a reduced curve with two $\mathbf{A}_1$ singularities for $i=2$, $3$ by \cite[\S 4.2, Table 2]{Pap22} (see also \cite{Hodge_book}). This case corresponds to (5).

Now let $(\lambda, x^I, x^J) = (\lambda_5,x_2x_3,x_0x_3)$. Then a general maximally unstable net $\N$ associated with the maximal set $N^{-}(\lambda, x^I, x^J)$ has three quadrics $Q_1$, $Q_2$ and $Q_3$ as generators such that $Q_1$ is smooth and $Q_2$, $Q_3$ are singular. Moreover, $Q_1\cap Q_3$ is a smooth curve, $Q_1\cap Q_2$ is a reduced curve with two $\mathbf{A}_1$ singularities (corresponding to a pencil with the Segre symbol $[(1,1),1,1]$) and $Q_2\cap Q_3$ is non-reduced. This case corresponds to (6).

If $(\lambda, x^I, x^J) =(\lambda_5,x_0x_3,x_0x_2)$, then a general maximally unstable net $\N$ associated with the maximal set $N^{\ominus}(\lambda, x^I, x^J)$ has three quadrics $Q_1$, $Q_2$ and $Q_3$ as generators such that $Q_1$ is smooth and $Q_2$, $Q_3$ are singular, and $Q_1\cap Q_i$ is a reduced curve with two $\mathbf{A}_1$ singularities for any $i=2$, $3$ (corresponding to a pencil with the Segre symbol $[2,2]$). This case corresponds to (7).

If $(\lambda, x^I, x^J) =(\overline{\lambda}_2,x_0x_3,x_0x_2)$, then a general maximally unstable net $\N$ associated with the maximal set $N^{-}(\lambda, x^I, x^J)$ has three quadrics $Q_1$, $Q_2$ and $Q_3$ as generators such that $Q_1$ is smooth and $Q_2 = H_1\cup H_2$, $Q_3$ are singular, and $Q_1\cap Q_2$ is a reduced curve with two $\mathbf{A}_1$ singularities (corresponding to a pencil with the Segre symbol $[(1,1),1,1]$), and $Q_1\cap Q_3$ is smooth. This case corresponds to (8).

If $(\lambda, x^I, x^J) =(\lambda_6,x_0x_2,x_0x_2)$, then a general maximally unstable net $\N$ associated with the maximal set $N^{-}(\lambda, x^I, x^J)$ has three quadrics $Q_1$, $Q_2$ and $Q_3$ as generators such that $Q_1$, $Q_2$ are smooth and $Q_3 = H_1\cup H_2$, and $Q_2\cap Q_3$ is a non-reduced curve which is the union of a double line and a third line (corresponding to pencil with Segre symbol $[(2,2)]$), and $Q_1\cap Q_3$ is a reduced curve with an $\mathbf{A}_3$ singularity (corresponding to a pencil with the Segre symbol $[(1,2),1]$). This case corresponds to (9).

If $(\lambda, x^I, x^J) =(\lambda_7,x_0x_3,x_2^2)$, then a general maximally unstable net $\N$ associated with the maximal set $N^{-}(\lambda, x^I, x^J)$ has three quadrics $Q_1$, $Q_2$ and $Q_3$ as generators such that $Q_1$ is smooth and $Q_2$, $Q_3$ are singular, and $Q_1\cap Q_i$ is a reduced curve with two $\mathbf{A}_1$ singularities for $i=2$, $3$ (corresponding to a pencil with the Segre symbol $[2,2]$). This case corresponds to (10).

If $(\lambda, x^I, x^J) =(\lambda_8,x_1x_3,x_1x_3)$, then a general maximally unstable net $\N$ associated with the maximal set $N^{-}(\lambda, x^I, x^J)$ has three quadrics $Q_1$, $Q_2$ and $Q_3$ as generators such that $Q_1$ is smooth and $Q_2=H_1\cup H_2$, $Q_3$ are singular, and $Q_1\cap Q_i$ is a reduced curve with two $\mathbf{A}_1$ singularities for $i=2$, $3$ (corresponding to a pencil with the Segre symbol $[(1,1),1,1]$). This case corresponds to (11).

To conclude, let $(\lambda, x^I, x^J) =(\lambda_9,x_1^2,x_0x_2)$. Then a general maximally unstable net $\N$ associated with the maximal set $N^{-}(\lambda, x^I, x^J)$ has three quadrics $Q_1$, $Q_2$ and $Q_3$ as generators such that $Q_1$, $Q_2$ are smooth quadrics with $Q_1\cap Q_2$ a smooth curve, and $Q_3$ is a singular quadric with $Q_i\cap Q_3$ a reduced curve with one $\mathbf{A}_3$ for $i=1$, $2$ (corresponding to a pencil with the Segre symbol $[(1,2),1]$). This case corresponds to (12).
                
The above list exhausts all possible cases given by \cite{Pap_code} up to changing the order of the product $A\times B \times C$ in the expression of maximal $N^{-}(\lambda,x^I,x^J)$. This proves the desired statement.
\end{proof}

The above results can also be applied to classify all semistable orbits, and in addition all polystable orbits, providing a more ``geometric'' description of the GIT quotient in the sense of Theorem \ref{unstable nets}. This extends the results of \cite{wall_theta}. Because this is beyond the scope of this paper, we omit it. 

We are mostly interested in the following corollary, which is one of the keys to deduce Corollary \ref{cor--final}.

\begin{cor}\label{cor: unstable nets discriminants}
    Let $\mathcal{N}$ be a net of quadrics in $\mathbb{P}^3$ with the associated discriminant quartic $\Delta(\N)\subsetneq\mathbb{P}^2$.
    Then the following hold.
    \begin{enumerate}
        \item If $\Delta(\N)$ is reduced and GIT unstable, then either it has a singular point of multiplicity at least three, or it consists of a cubic and an inflectional tangent line;
    \item if $\Delta(\N)$ is GIT unstable and not reduced, then it is not a double smooth conic.
    
\end{enumerate}
In particular, if $\Delta(\N)$ is GIT semistable, so is $\N$.
\end{cor}

\begin{proof}
   
First, note that $(1)$ and $(2)$ follow from the last assertion and \cite[Section 1.12]{mumford}.
Now, we claim that to show the last assertion for any unstable net, it suffices to check the assertions $(1)$ and $(2)$ only for general maximally unstable nets.
Indeed, let $T$ be a smooth affine curve with a closed point $0\in T$, and let $\{\N_t\}_{t\in T}$
be a one-parameter family of nets of quadrics in $\mathbb{P}^3$.
Here, we can choose a family of polynomials $Q_i(t)$ for $i=1,2$ and $3$ that depend algebraically on $t$ and generate $V_{\N_t}$. Then the associated discriminant quartic $\Delta(\N_t)$ is given by an equation $F_t(x,y,z):=\det\bigl(xQ_1(t)+yQ_2(t)+zQ_3(t)\bigr)$. Since the determinant is a polynomial in the entries of the matrix
$xQ_1(t)+yQ_2(t)+zQ_3(t)$, it follows that the coefficients of $F_t(x,y,z)$ depend algebraically on $t$.
If $F_0(x,y,z)\ne0$, then $\{\Delta(\N_t)\}_{t\in T}$ defines a one-parameter family of plane quartics. If $\Delta(\N_t)$ is further GIT unstable for any general $t\in T$, then so is $\Delta(\N_0)$ by the openness of the semistable locus $\mathcal{X}_{2,2,3}^{\mathrm{ss}}$. 
In particular, since by Theorem \ref{thm: pap paper thm} all unstable orbits are given as degenerations of the maximally unstable nets given in Theorem \ref{unstable nets}, it suffices to only consider and classify the discriminant quartic curves for all maximally unstable orbits given in Theorem \ref{unstable nets} by \cite[Section 1.12]{mumford}.

We classify all of them explicitly. For instance, if we let $(\lambda, x^I, x^J) = (\lambda_1, x_0^2,x_0^2)$, then the maximal set $N^{-}(\lambda, x^I, x^J)$ gives a general maximally unstable net $\mathcal{N} $ with three generators $Q_1$, $Q_2$ and $Q_3$ such that $Q_1$, $Q_2$ are smooth and $Q_3 = 2H$ is a double hyperplane. By a calculation and the description of the supports of $f_1$, $f_2$ and $f_3\in V_{\N}$ representing $Q_1$, $Q_2$ and $Q_3$, the discriminant curve $\Delta(\N)$ is given by equation $g_4(x,y)  +  zg_3(x,y)$, where $g_4$ is a homogeneous polynomial of degree four and $g_3$ is a polynomial of degree three in variables $x$ and $y$. In particular, $\Delta(\N)$ has a singularity $\mathbf{D}_4$ or worse at the point $(0:0:1)$, and is thus GIT unstable. A similar argument applied to each and every maximally unstable orbit given in Theorem \ref{unstable nets} yields a description of the discriminant curves for all families of maximally unstable nets. This description is summarized in Table \ref{tab:deltaN-singularities}. 
\begin{center}
    \begin{table}[ht]
\centering 
\renewcommand{\arraystretch}{1.25}
\begin{tabular}{|c|c|c|c|}
\hline
\ref{unstable nets}.(-)& $N^{-}(\lambda, x^I, x^J)$ & Description of $\Delta(\mathcal N)$ & Type of sing. 
\\
\hline
\ref{unstable nets}.(1)& $N^{-}(\lambda_1, x_0^2,x_0^2)$
& $g_4(x,y) + z g_3(x,y)$
& $\text{at least }\mathbf{D}_4$ 
\\
\hline
\ref{unstable nets}.(2)& $N^{-}(\lambda_1, x_0x_3,x_0^2)$
& $x^2 g_2(x,y,z)$ (not double conic)
& non-reduced  
\\
\hline
\ref{unstable nets}.(3)& $N^{-}(\lambda_2, x_1x_3, x_0x_3)$
& $g_1(x,y)^2g_2(x,y,z)$
& non-reduced 
\\
\hline
\ref{unstable nets}.(4) &$N^{-}(\lambda_3, x_1^2, x_0x_1)$
& $x^2 g_2(x,y,z)$ (not double conic)
& non-reduced  
\\
\hline
\ref{unstable nets}.(5)& $N^{-}(\lambda_4, x_2^2,x_0^2)$
& $x^2 g_2(x,y,z)$ (not double conic)
& non-reduced 
\\
\hline
\ref{unstable nets}.(6)& $N^{-}(\lambda_5,x_2x_3,x_0x_3)$
& $x(xg_2(x,y,z)+z^3)$
& $\text{at least }\mathbf{A}_5$
\\
\hline
\ref{unstable nets}.(7)& $N^{-}(\lambda_5,x_0x_3,x_0x_2)$
& $x^2 g_1(x,y,z)^2$ (not double conic)
& non-reduced 
\\
\hline
\ref{unstable nets}.(8)& $N^{-}(\overline{\lambda}_2,x_0x_3,x_0x_2)$&
$x(g_3(x,y)+xyz+\alpha zx^2+\beta z^2x)$
& $\text{at least }\mathbf{A}_5$
\\
\hline
\ref{unstable nets}.(9)& $N^{-}(\lambda_6,x_0x_2,x_0x_2)$&
$g_4(x,y) + xzg_2(x,y)$
& $\text{at least }\mathbf{D}_4$ 
\\
\hline
\ref{unstable nets}.(10)& $N^{-}(\lambda_7,x_0x_3,x_2^2)$
& $x^2 g_1(x,y,z)^2$ (not double conic)
& non-reduced 
\\
\hline
\ref{unstable nets}.(11)& $N^{-}(\lambda_8,x_1x_3,x_1x_3)$
& $x^2 g_2(x,y,z)$ (not double conic)
& non-reduced 
\\
\hline
\ref{unstable nets}.(12)& $N^{-}(\lambda_9,x_1^2,x_0x_2)$
& $g_4(x,y) + z g_3(x,y)$ & $\text{at least }\mathbf{D}_4$
\\
\hline
\end{tabular}
\caption{Types of singularities of $\Delta(\mathcal N)$ for maximally unstable nets $\N$ associated with $N^{-}(\lambda, x^I, x^J)$.}\label{tab:deltaN-singularities}
\end{table}
\end{center}

Notice that the entries for \ref{unstable nets}.(6) and \ref{unstable nets}.(8) correspond to the quartic curve, which is the union of a cubic with an inflectional line. Therefore, Table \ref{tab:deltaN-singularities} and Theorem \ref{unstable nets} show that the assertions $(1)$ and $(2)$ hold for the general maximally unstable nets. This completes the proof.
\end{proof}

\bibliographystyle{plain}
\bibliography{ref}

\end{document}